\documentclass[12pt,leqno]{article}

\usepackage{amsmath}
\usepackage{amsthm}
\usepackage{amssymb}

\title{Regularity in the local CR embedding problem}
\author{Xianghong Gong\footnote{Partially supported by NSF grant
 DMS-0705426.}\ \,  and S. M. Webster}

\newcommand{\gra}{\alpha}
\newcommand{\grb}{\beta}
\newcommand{\grd}{\delta}          
\newcommand{\gre}{\epsilon}
  
\newcommand{\grg}{\gamma}          

\newcommand{\grk}{\kappa}
\newcommand{\grl}{\lambda}         
\newcommand{\grm}{\mu}

\newcommand{\grr}{\rho}
\newcommand{\grs}{\sigma}          \newcommand{\grS}{\Sigma}
\newcommand{\grt}{\tau}
          
\newcommand{\grx}{\xi}             
\newcommand{\gry}{\eta}

\newcommand{\grph}{\phi}           \newcommand{\grPh}{\Phi}

\newcommand{\grch}{\chi}
\newcommand{\grps}{\psi}

\newcommand{\la}{\langle}
\newcommand{\ra}{\rangle}

\newcommand{\del}{\partial}

                           \newcommand{\bC}{\mathbf{C}}

\newcommand{\beq}{\begin{equation}}
\newcommand{\eeq}{\end{equation}}

\newtheorem{theorem}{Theorem}[section]
\newtheorem{lemma}[theorem]{Lemma}
\newtheorem{prop}[theorem]{Proposition}

\newcommand{\barr}{\overline}

\begin{document}

\maketitle

\noindent\textbf{Abstract}.  We consider a formally integrable, strictly
pseudoconvex CR manifold $M$ of hypersurface type, of dimension
$2n-1\geq7$.  Local CR, i.e. holomorphic, embeddings of $M$ are known
to exist from the works of Kuranishi and Akahori.  We address the
problem of regularity of the embedding in standard H\"older spaces $C^{a}(M)$,
$a\in\mathbf{R}$. If the structure of $M$ is of class $C^{m}$,
$m\in\mathbf{Z}$, $4\leq m\leq\infty$, we construct a local CR embedding
near each point of $M$.  This embedding is of class $C^{a}$, for every
$a$, $0\leq a < m+(1/2)$.  Our method is based on Henkin's local
homotopy formula for the embedded case, some very precise estimates
for the solution operators in it, and a substantial modification of a
previous Nash-Moser argument due to the second author.



\tableofcontents

%\newpage

\vspace{3ex}

%00000000000000000000000000000000000000000000000000000000000000000000000000000

\setcounter{section}{0}
\setcounter{equation}{0}

\noindent\textbf{INTRODUCTION}.  In this paper we prove the following
\textbf{theorem:}  Let $M$ be a formally integrable, strongly
pseudoconvex CR manifold of differentiability class $C^{m}$,
$m\in\mathbf{Z}$, and dimension $2n-1\geq 7$.  Then, near each point of
$M$, there exists a local CR embedding $Z$ into $\mathbf{C}^{n}$.
This embedding $Z$ is of H\"older class $C^{a}$, for every $a$,
$0\leq a < m+1/2$.

We state the result in a more precise form below.  Locally, we take
$U\subset \mathbf{R}^{2n-1}$ to be a neighborhood of the origin.  The
CR structure is given by $n-1$ complex vector fields, $X_{\gra}$,
$1\leq \gra \leq n-1$ on $U$, which together with their complex conjugates
$X_{\barr{\gra}}$ are pointwise independent over $\mathbf{C}$.  The
Lie brackets $[X_{\gra},X_{\grb}]$ are linear combinations of the
$X_{\gra}$'s, which is the integrability condition.  The brackets
$i[X_{\gra},X_{\barr{\grb}}]$, modulo the $X_{\gra},X_{\barr{\gra}}$'s
give the Hermitian Levi form, which is assumed to be positive definite.
The CR embedding $Z=(z^{1},\dots ,z^{n})$, is to be given by $n$
independent local complex functions near $0$, which satisfy the
overdetermined system of first-order, linear partial differential
equations
\begin{equation}
  X_{\barr{\gra}}z^{j}=0, \; \; 1\leq \gra \leq n-1, \; 1\leq j \leq n.
\end{equation}
Our main result is the following.
\begin{theorem} Let the coefficients of the vector-fields $X_{\gra}$
be of class $C^{m}$, $m\in\mathbf{Z}$, $4\leq m \leq\infty$,and $2n-1\geq 7$.
Then there exist $n$ independent solutions $z^{j}$ to the above system,
which embed some neighborhood of $0$ as a strictly pseudoconvex real
hypersurface $M^{2n-1}$ in $\mathbf{C}^{n}$.  These functions $z^{j}$ are of
H\"older class $C^{a}$, for all $a$, $0\leq a < m+1/2$.
\end{theorem}

There are counterexamples to the existence of local CR embeddings in the
case $2n-1=3$, due to Nirenberg [20], while the case $2n-1=5$ is still
unresolved, to our knowledge.  There are also counterexamples in the case
of Levi nondegenerate structures of mixed signature [11].

The lower bound on $m$ can undoubtedly be improved. A reasonable
conjecture would be that the theorem holds with $m\geq2$. In fact,
for many results in this work, we may take the $X_{\gra}$ in the
H\"older class $C^{m}$, $m\in\mathbf{R}$, $3<m\leq\infty$.  Then we can get a
CR-embedding of class $C^{a}$, if $0\leq a<m$.
See proposition 12.1 below.

Perhaps more interestingly, it seems likely that one might achieve
$a=m + 1/2$ for ours, or for some other solution $Z$.
This reminds us of  the situation of the fundamental (1/2)-estimate for the
$\barr{\del}$-problem on strictly pseudoconvex domains (see Range [21]).
Recall briefly that Kerzman [12] and Stein  showed that one
cannot gain more than H\"older 1/2 in this problem, and Kerzman established
a gain of $<1/2$.  After much work, Henkin and Romanov [7], [9] proved the
existence of solutions gaining precisely 1/2 derivative, by kernel methods.

Henkin's construction [8] of solution operators for the local
tangential Cauchy-Riemann equations on a strictly pseudoconvex real
hypersurface $M$, together with the detailed estimates in [6], form
key ingredients in the proof of the above theorem.  Since these operators
do not regain a full derivative, we must introduce a smoothing process.
Thus, the other main ingredient is a Nash-Moser implicit function
theorem [16], [17], [18].

The CR embedding problem now has a long history.  The local result
was first conjectured by Kohn [13], who also founded the analysis
of CR manifolds using Hilbert space methods.  Theorem 0.1 is an
analogue of the Newlander-Nirenberg theorem (see [20]) for
integrable almost-complex structures.  That venerable result now
has many different proofs.  We mention only the one in [22], the
method of which is particularly relevant to the argument given
here.  It motivated the proof in [24] of a much less precise
version of theorem 0.1.  The work of Ma and Michel [15] greatly
reduced the derivative loss in that argument.  We also mention
another simpler but instructive model, the integrability problem
for CR vector bundles.  In [5] we were able to eliminate
completely the previous complicated Nash-Moser techniques and give
a simple and sharp argument based on the KAM method.  To date,
however, all known proofs of local CR embedding involve some kind
of difficult Nash-Moser argument.

In case the CR manifold $M^{2n-1}$ is also compact, and $2n-1\geq5$,
Boutet-de-Monvel has given a linear proof (of local embedding) [2],
based on Kohn's estimates [13].  By far the most important step
toward theorem 0.1 was taken by Kuranishi [14] for $2n-1\geq9$.  This
was extended to $2n-1\geq7$ by Akahori [1].  Catlin has given another
proof of even more general results [3] based on his method of extending the
CR structure to an integrable almost complex structure.  These results
are all carried out in $C^{\infty}$, using $\barr{\del}$-Neumann-type
methods.

The scheme of proof of theorem 0.1 is similar to that developed in [24],
but with some substantial and significant changes.  We start with a smooth
approximate CR (or holomorphic) embedding and modify it to make it more
nearly holomorphic. We iterate this process, generating a
sequence of smooth embeddings, which converge in suitable norm to a
CR embedding with the stated regularity.

In section one we make an initial normalization of the complex vector
fields $X_{\gra}$ at $0$. This produces an approximate embedding onto
a real hyperquadric in $\mathbf{C}^{n}$.  Instead of using higher order
Taylor polynomial arguments, we apply non-isotropic dilation to get
approximate CR embeddings with arbitrarily small error.  This obviates
much of the spurious derivative loss encountered in [24].

In section 2 we recall the tangential CR operator $\barr{\del}_{M}$
and local Henkin homotopy formula, with solution operators $P$, $Q$  for
a normalized strictly pseudoconvex real hypersurface $M\subset\mathbf{C}^{n}$.
By normalized we mean that $M$ is suitably approximated to second
order at the origin by a real hyperquadric.  Then in section 3 we
indicate our basic procedure, $Z\mapsto Z_{*}=Z + F$, to make the embedding
more nearly holomorphic. For this the error term $X_{\barr{\gra}}(Z+F)$ is
decomposed into a sum of 4 terms.  This procedure destroys the normalization
at $0$, so section 4 is needed to give a precise method for renormalization.
This uses the first order Taylor polynomial of $F$. The inverse mapping
theorem is then applied to the map $f$ gotten by projecting
$Z\mapsto Z_{*}$ to the real hyperplane Re$z^{n}=0$. The map $f$
reparametrizes the CR structure, and $Z_{*}\circ f^{-1}$ is the new
approximate CR embedding.

This procedure is rather lengthy, but it is one for which we can make precise
estimates.  These are in terms of standard H\"older norms, which we recall
in section 5.  We make use of elementary properties of these norms
in sections 5 and 6 to get estimates in terms of $F$.  These
estimates are roughly of two types.  The \lq\lq coarse\rq\rq\ estimates
will be used for controlling the growth of higher order derivatives during
iteration.  The \lq\lq fine\rq\rq\ estimates will be used to get rapid
convergence in some lower order H\"older norm.

In section 7 we introduce a standard smoothing operator $S_{t}$ and
give some well known estimates for it.  Moreover, we make a precise estimate
for the Friedrichs commutator $[S_{t},\barr{\del}_{M}]$.  Also in section
7 we quote from [6] the needed estimates for the Henkin solution
operators $P$, $Q$.  In section 8 we give estimates for $F$, and use them
to refine our previous estimates. Section 9 contains the central estimates,
the  4-term estimates for $X_{\barr{\gra}}(Z+F)$.

In sections 10, 11 we establish the main inductive hypotheses and construct the
sequence of approximate CR embeddings.  It is shown to converge
in a low derivative norm  to a CR embedding in section 12, provided
$m>3$. This inequality results from the choice of certain parameters
in the argument, and is probably not optimal. In
section 13 we use the interpolation inequality for H\"older norms, the
rapid convergence in low norm, and some delicate estimates controlling
the possible growth of higher order derivatives, to gain increased
regularity for the limiting real hypersurface $M_{\infty}$ and its
embedding $Z_{\infty}$ into $\mathbf{C}^{n}$.  In section 14 we show
convergence of the infinite composition of the  projected maps $f$ to
a map $\tilde{f}_{\infty}$. The solutions to the system (0.1) are then given
by the components $z^{j}$ of $Z_{\infty}\circ \tilde{f}_{\infty}$.  This
will prove theorem 0.1.

%111111111111111111111111111111111111111111111111111111111111111111111111111111

\section{Initial normalization. Non-isotropic dilation.}
\setcounter{equation}{0}

We consider $n-1$ complex vector fields $X_{\gra}$, $1\leq \gra \leq n-1$,
in a neighborhood of $0\in\mathbf{R}^{2n-1}$, of smoothness class $C^{m}$,
$m\in\mathbf{Z}$,  $m\geq2$, which, together with their complex conjugates
$X_{\barr{\gra}}=\barr{X_{\gra}}$, are pointwise linearly independent over
the complex numbers.  Under the integrability and non-degeneracy conditions,
we normalize them at $0$ to get a second order approximate CR embedding onto
a real hyperquadric.  Then we use non-isotropic dilation to make the
error as small as we please in $C^{m}$-norm.

\begin{lemma}
  After a polynomial change of the coordinates $(z^{\gra},x^{n})$ on
$\mathbf{R}^{2n-1}$, we may achieve
\begin{equation}
   X_{\gra} = \del_{\gra} + \tilde{A}_{\gra}^{\;\barr{\grb}}\del_{\barr{\grb}}
    + \tilde{B}_{\gra}\del_{x^{n}},\; \; \;
   \tilde{B}_{\gra} = i\grd_{\gra\barr{\grb}}z^{\barr{\grb}} + B_{\gra}^{*},
\end{equation}
where $\tilde{A}_{\gra}^{\;\barr{\grb}}\equiv B_{\gra}^{*}
  \equiv O(2)\equiv O(|(z',x^{n})|^{2})$, and the coefficients are of
class $C^{m}$.
\end{lemma}
\begin{proof}
We only sketch the proof, since the principle is well known.
We first choose linear coordinates so that
$X_{\gra}(0) = \del_{\gra}\equiv \del/\del z^{\gra}$. Then a frame change
achieves the form (1.1) with $\tilde{A}_{\gra}^{\;\barr{\grb}}(0)=0$,
$\tilde{B}_{\gra}(0)=0$.  For such a frame the integrability condition
is equivalent to the Lie brackets $[X_{\gra},X_{\grb}]=0$, or
\begin{equation}
  X_{\gra}\tilde{A}^{\;\barr{\grg}}_{\grb} =
  X_{\grb}\tilde{A}^{\;\barr{\grg}}_{\gra} , \; \;
  X_{\gra}\tilde{B}_{\grb} = X_{\grb}\tilde{B}_{\gra} .
\end{equation}
The Levi form is the hermitian matrix $g_{\gra\barr{\grb}}$ defined by
\begin{eqnarray}
     [X_{\gra},X_{\barr{\grb}}] &  \equiv  &
     -ig_{\gra\barr{\grb}}\del_{x^{n}} , \; \mbox{mod}
     \{ X_{\grg},X_{\barr{\grg}} \}, \\
     g_{\gra\barr{\grb}}(0) & = & i( X_{\gra}\tilde{B}_{\barr{\grb}} -
      X_{\barr{\grb}}\tilde{B}_{\gra} )(0). \nonumber
\end{eqnarray}
It is assumed to be positive definite. After a change of the frame
$X_{\gra}$, we may assume $g_{\gra\barr{\grb}}(0)=2\grd_{\gra\barr{\grb}}(0)$.
The symmetry in $\gra,\grb$ in (1.2)
allows us to remove certain terms using Taylor polynomials in $z'=(z^{\gra})$
and get the normalization along the $x^{n}$-axis.  We replace this
normalizing transformation by its Taylor polynomial at $0$, to avoid losing
derivatives.  For more details see the first
section of [24].\end{proof}

We now consider approximate CR, or holomorphic, embeddings of the form
\begin{equation}
  Z = (z',z^{n}), \; \; z^{n} = x^{n} + iy^{n},\; \; y^{n}=y^{n}(z',x^{n}),
\end{equation}
onto a smooth strictly pseudoconvex real hypersurface
$M\subset\mathbf{C}^{n}$, of the form
\begin{equation}
   M : \; r = 0,\; \; r(z) = -y^{n} + |z'|^{2} + h(z',x^{n}),
\end{equation}
where $h\equiv O(2)$ is a $C^{\infty}$-smooth real function defined on a
neighborhood $D$ of $0$, and vanishing to order 2 at $0$,
\begin{equation}
 0\in D\subseteq\mathbf{R}^{2n-1}=\{y^{n}=0\}\subset\mathbf{C}^{n}.
\end{equation}
The hypersurface and embedding will be referred to as \emph{normalized}.

A basis of complex tangential (1,0)-vector fields to $M$ is given by
\begin{equation}
  Y_{\gra} = \del_{\gra} - (r_{\gra}/r_{n})\del_{n}\equiv
             \del_{\gra} - (r_{\gra}/2r_{n})\del_{x^{n}},
\end{equation}
the latter being the representation on $\mathbf{R}^{2n-1}$.
The given vector fields $X_{\gra}$ and their complex conjugates
$X_{\barr{\gra}}$, now considered on $D$, will be \emph{adapted}
to the embedding $Z$ by the condition $X_{\gra}z^{\grb}=\grd_{\gra}^{\;\grb}$,
which determines them uniquely.  We then write them as
\begin{eqnarray}
  X_{\gra} & = & Y_{\gra} + A_{\gra}^{\;\barr{\grb}}Y_{\barr{\grb}} +
      B_{\gra}\del_{x^{n}} , \\
  X_{\barr{\gra}}z^{\grb} & = & A_{\barr{\gra}}^{\;\grb} \, , \, \; \;
  X_{\barr{\gra}}z^{n} \, = \,    ( 1 + i\del_{x^{n}}h)B_{\barr{\gra}} -
   A_{\barr{\gra}}^{\;\grb}(r_{\grb}/r_{n}).
\end{eqnarray}
It follows that the embedding is nearly holomorphic, to the extent that
the coefficients
$A_{\barr{\gra}}^{\;\grb}= \barr{A_{\gra}^{\;\barr{\grb}}}$,
and $B_{\barr{\gra}}=\barr{B_{\gra}}$ are small.

With the normalization (1.1), we achieve the embedding (1.4) with $h=0$
in (1.5) and with vector fields of the form (1.7), (1.8), where
$A_{\gra}^{\;\barr{\grb}} = \tilde{A}_{\gra}^{\;\barr{\grb}}$ and
$ B_{\gra} = B_{\gra}^{*} +
      iA_{\gra}^{\;\barr{\grb}}(r_{\barr{\grb}}/2r_{\barr{n}})$.  Thus,
the embedding is onto a real hyperquadric $M$, with $X_{\barr{\gra}}Z=O(2)$.

The real hyperquadric $M$, and the real hyperplane $y^{n}=0$, are
both invariant by the family of non-isotropic dilations,
$f_{\grr}(z',z^{n}) = (\grr z',\grr^{2}z^{n})$,
$f_{\grr}:D_{s}^{0}\rightarrow D_{\grr s}^{0}$, where
$D_{s}^{0}=\{|z'|^{4}+(x^{n})^{2} < s^{4}\}\subset\mathbf{R}^{2n-1}$.
We take $s$ small, $s\grr=1$, replace $Z$ by $f_{\grr}\circ Z$, and the
vector fields $X_{\gra}$ by
$X^{(\grr)}_{\gra}=\grr^{-1}(f_{\grr})_{*}(X_{\gra})$. Then
\begin{eqnarray}
   X^{(\grr)}_{\gra} & = & Y_{\gra}
   + A_{\gra}^{(\grr)\barr{\grb}} Y_{\barr{\grb}}
    + B^{(\grr)}_{\gra}\del_{x^{n}},\\
  A_{\gra}^{(\grr)\barr{\grb}} & = &
   A_{\gra}^{\;\barr{\grb}}(\grr^{-1}z',\grr^{-2}x^{n}), \;\;
  B^{(\grr)}_{\gra} =
   \grr B_{\gra}(\grr^{-1}z',\grr^{-2}x^{n}).
\end{eqnarray}
(Note that a general dilation $x\mapsto\grr x =(\grr_{1}x_{1},\ldots ,
\grr_{n}x_{n})$ transforms a vector field $\grS\grx_{j}(x)\del_{x_{j}}$
to $\grS\grx_{j}(\grr^{-1}x)\grr_{j}\del_{x_{j}}$.)
Taking any number $k\leq m$ of $(z',x^{n})$-derivatives, we see that
\begin{equation}
  \del^{k}A_{\gra}^{(\grr)\barr{\grb}} = O(\grr^{-2}), \;\;
  \del^{k}B_{\gra}^{(\grr)} = O(\grr^{-1}),
\end{equation}
hence $X_{\barr{\gra}}^{(\grr)}[f_{\grr}\circ Z]=O(\grr^{-1})$,
in $C^{m}$-norm, as $\grr\rightarrow\infty$. Thus, by shrinking our
original domain in $\mathbf{R}^{2n-1}$, dilating it to unit size,
and dropping the notation $(\grr)$ in (1.10), we achieve the following lemma.
\begin{lemma}
Near any point, taken as the origin, there are approximate holomorphic
embeddings $Z$, of the form (1.4) defined on the unit ball $B_{1}$ in
$\mathbf{R}^{2n-1}$, for the given CR structure of class $C^{m}$, $m\geq2$.
Basis vector fields $X_{\gra}$ for the structure satisfy
$X_{\barr{\gra}}Z=O(2)$, and may be taken in the form (1.8), with
$X_{\barr{\gra}}Z$ as small as we please in the $C^{m}(B_{1})$-norm.
\end{lemma}

%22222222222222222222222222222222222222222222222222222222222222222222222222222
\section{Local real hypersurfaces in $\mathbf{C}^{n}$.}
\setcounter{equation}{0}

In this section we consider a local, normalized, strictly pseudoconvex
real hypersurface $M\subset\mathbf{C}^{n}$.  First we determine the
relevant geometric properties of $M$, under some further conditions.
Then we recall the local tangential Cauchy-Riemann complex, and the local
homotopy formula first given by Henkin.

To establish more  precisely the needed properties of $M$ (1.5), we assume
that the domain $D$ (1.6) is a convex neighborhood of $0$.  We further define
$x = (z',x^{n})$, $|x|^{2} = (x^{n})^{2} + z'\cdot\barr{z}'$, and
\begin{eqnarray}
  \grps(x) & = &  |x|^{2} + h(x) = (x^{n})^{2} + |z'|^{2} + h(z',x^{n}), \\
  D_{\grr} & \equiv & D_{\grr}(h) \, = \{x\in D \, | \, \grps(x)\leq\grr^{2}\}.
\end{eqnarray}

By Taylor's formula, $|h(x)|\leq c_{2}|x|^{2}$ on $D$, where
$c_{2} = \|h\|_{C^{2}(D)}$.  We assume that $c_{2} < 1$ so that
\begin{eqnarray}
   (1-c_{2})|x|^{2} & \leq \grps(x) & \leq \, (1+c_{2})|x|^{2}, \\
    B(\grr/\sqrt{1+c_{2}})\cap D & \subseteq D_{\grr} & \subseteq \,
    B(\grr/\sqrt{1-c_{2}})\cap D ,
\end{eqnarray}
where $B(\grr)$ denotes the ball of radius $\grr$ centered at $0$.
In particular, $D_{\grr}$ is contained in the interior of $D$, if
$\grr > 0$ is sufficiently small, which we now also assume.

To investigate the smoothness of $\del D_{\grr}$, suppose $0=d\grps(x)$.
Then $2|x|^{2}=-x^{j}h_{j}(x)$, so that $2|x|^{2}\leq c_{2}|x|^{2}$;
hence $x=0$. It follows that $\del D_{\grr}$ is smooth for small $\grr>0$.
Then $D_{\grr}$ will be strictly convex, if the Hessian $(\grps_{ij})$
is positive definite on $D$. But this also holds, since
$\grps_{ij}(x)v^{i}v^{j}=2|v|^{2} + h_{ij}(x)v^{i}v^{j}
\geq (2-c_{2})|v|^{2}$.

Finally, we estimate the distance between $\del D_{\grr}$ and
$\del D_{\hat{\grr}}$, where $\hat{\grr}=\grr(1-\grs)$, $0<\grs<1$,
assuming these to be compact, smooth, and strictly convex.  Let
$x_{0}\in\del D_{\hat{\grr}}$ and $x_{1}\in\del D_{\grr}$ satisfy
$|x_{0}-x_{1}|=$ dist$(\del D_{\hat{\grr}},\del D_{\grr})$.  Then the
segment $[x_{0},x_{1}]$ is orthogonal to the two tangent planes, and
$|x_{1}|>|x_{0}|$, since $D_{\hat{\grr}}$ is a convex neighborhood of $0$.
By the Schwarz inequality and the fact that
$|dh(x)|\leq c_{2}|x_{1}|$, for $x$ on the segment $[x_{0},x_{1}]$,
we have
\begin{equation}
\begin{array}{rcl}
  \grr^{2} - \hat{\grr}^{2} & = & \la x_{1}-x_{0},x_{1}\ra +
   \la x_{0},x_{1}-x_{0}\ra +
   h(x_{1}) - h(x_{0}) \\
    & \leq & |x_{1}-x_{0}|(|x_{0}| + |x_{1}| ) +
     |x_{1}-x_{0}| c_{2}|x_{1}| \\
    & \leq & (2+c_{2})|x_{1}-x_{0}|\cdot|x_{1}|.
\end{array}
\end{equation}
But by (2.9), $\grr^{2} \geq(1 - c_{2})|x_{1}|^{2}\geq (1/2)|x_{1}|^{2}$,
provided $c_{2}<1/2$.  Then
\begin{equation}
  \grr\grs \, \leq \, \hat{c}\cdot \mbox{dist}
         (\del D_{\grr(1-\grs)},\del D_{\grr}),
\end{equation}
where $\hat{c}= \sqrt{2}(2+c_{2})\leq3\sqrt{2}$.

This proves the following.
\begin{lemma}
  There is a constant $\grg_{0}$, $0 < \grg_{0} < 1/2$, such that the
following holds. Let $D$ in (1.6) be a convex neighborhood of $0$, and
let the smooth real function $h$ as above satisfy
\begin{equation}
  h = O(|x|^{2}), \; \;   \|h\|_{C^{2}(D)} < \grg_{0}.
\end{equation}
Suppose that $\grr_{0}>0$ is so small that
$D_{\grr_{0}}$ in (2.2) is contained in the interior of $D$. Then
for each  $0<\grr\leq\grr_{0}$, the domain $D_{\grr}$ is compact, smoothly
bounded, strictly convex, and
\begin{equation}
  B(\sqrt{2/3}\grr)\,\subseteq \, D_{\grr}\,\subseteq \,
  B(\sqrt{2}\grr).
\end{equation}
For $0<\grr\leq\grr_{0}$ and
$0<\grs<1$, the inequality (2.6) holds.
\end{lemma}

We now consider the vector fields $Y_{\gra}$ (1.7), and $X_{\grs}$
(1.8), on $D$ and the corresponding tangential CR-complexes.
For functions $f$ on $M$ (or on $D$), we define
\begin{equation}
  \barr{\del}_{M}f = \sum Y_{\barr{\gra}}fdz^{\barr{\gra}}, \; \;
  \barr{\del}_{X}f = \sum X_{\barr{\gra}}fdz^{\barr{\gra}}.
\end{equation}
For a tangential (0,q)-form on $M$,
$\grph = \sum \grph_{\barr{A}}d\barr{z}^{\barr{A}}$, in standard multi-index
notation, $A=(\gra_{1},\ldots ,\gra_{n-1})$, $1\leq\gra_{j}\leq n-1$,  we set
\begin{equation}
  \barr{\del}_{M}\grph =
  \sum \barr{\del}_{M}\grph_{\barr{A}}\wedge d\barr{z}^{\barr{A}},  \; \; \;
  \barr{\del}_{X}\grph =
  \sum \barr{\del}_{X}\grph_{\barr{A}}\wedge d\barr{z}^{\barr{A}}.
\end{equation}
From the form of the vectors $X_{\gra}, Y_{\gra}$, and their respective
integrability conditions, it follows that $[X_{\gra},X_{\grb}]=0$, and
$[Y_{\gra},Y_{\grb}]=0$.   This gives
\begin{equation}
  (\barr{\del}_{M})^{2} = 0, \; \; \; (\barr{\del}_{X})^{2} = 0.
\end{equation}

For $0<\grr\leq\grr_{0}$ as in the lemma, we denote by $M_{\grr}$ the
part of $M$ lying over $D_{\grr}$.  Thus,
\begin{equation}
  M_{\grr} \, = \, M \cap \{(z',z^{n})\in\mathbf{C}^{n} \, | \,
 (x^{n})^{2} + y^{n} \, \leq \, \grr^{2}  \}.
\end{equation}
It is a compact, smoothly bounded domain on the strictly
pseudoconvex real hypersurface $M$.  Its main property is that it is
cut out by a real function of the holomorphic function, $z^{n}$.
For such domains we have the local Henkin homotopy formula [8], [23], [6],
\begin{equation}
   \grph = \barr{\del}_{M}P\grph + Q\barr{\del}_{M}\grph,
\end{equation}
for (0,q)-forms $\grph$ as above,
where $P$, $Q$ involve certain integral operators over $M_{\grr}$ and
its boundary.  This formula is valid only for $0 < q < n-2$.  We use it
in the case $q=1$.  This requires dim $M =2n-1 \geq 7$, thus excluding the
five dimensional case.  This is the only point in our argument where
this condition is required.  In the 5 dimensional case there is an extra
term added to the right hand side of (2.13).  As of this writing, it remains
unclear whether such a more general homotopy formula can be used.  See [24]
and Nagel-Rosay [19].

%333333333333333333333333333333333333333333333333333333333333333333333333333333
\section{Alteration of the embedding.}
\setcounter{equation}{0}
Given an embedding $Z$ such that the error
$\|\barr{\del}_{X}Z\|= \sum\|X_{\barr{\gra}}Z\|$ is
\lq\lq small\rq\rq , we wish to perturb it to one
$Z_{*}(z',x^{n})=Z(z',x^{n})+F(z',x^{n})$, so that the new error
$\|\barr{\del}_{X}Z_{*}\|$ is \lq\lq smaller\rq\rq.  These terms as well
as the norms will be made precise later.

We assume the results of the previous section, and apply the homotopy
formula (2.13) component-wise to
$\grph = \barr{\del}_{X}Z$, to get
\begin{eqnarray}
  \barr{\del}_{X}Z_{*} & = & \barr{\del}_{X}Z + \barr{\del}_{M}F +
   (\barr{\del}_{X}- \barr{\del}_{M})F \\ &  = &
   \barr{\del}_{M}(P\barr{\del}_{X}Z + F)+Q(\barr{\del}_{M}\barr{\del}_{X}Z)
   +(\barr{\del}_{X}- \barr{\del}_{M})F. \nonumber
\end{eqnarray}
By Newton's method we would take $F=-P\barr{\del}_{X}Z$.  The integrability
condition gives $\barr{\del}_{M}\barr{\del}_{X} =
(\barr{\del}_{M} -\barr{\del}_{X})\barr{\del}_{X}$. The coefficients of
$\barr{\del}_{X}- \barr{\del}_{M}$ are dominated by $\|\barr{\del}_{X}Z\|$,
by (1.8), (1.9). Thus, formally we
would have $\|\barr{\del}_{X}Z_{*}\| \leq \|\barr{\del}_{X}Z\|^{2}$.
However, the operator $P$ may not fully regain the derivative lost in
applying $\barr{\del}_{X}$ to $Z$.  Thus we introduce the smoothing
operator $S_{t}$ of section 7, and set
\begin{equation}
  F = - S_{t}P\barr{\del}_{X}Z.
\end{equation}
Then
\begin{equation}
 \barr{\del}_{X}Z_{*} = \barr{\del}_{M}(I - S_{t})P\barr{\del}_{X}Z +
  Q\barr{\del}_{M}\barr{\del}_{X}Z + (\barr{\del}_{X}- \barr{\del}_{M})F.
\end{equation}
This 3-term decomposition can be made to work [24]; however, its main
deficiency seems to be that there is no smoothing in the $Q$ term.
Thus, we introduce the commutator, applying $S_{t}$ to the coefficients of
a differential form,  and write
\begin{equation}
 \barr{\del}_{X}Z_{*} = [\barr{\del}_{M},(I - S_{t})]P\barr{\del}_{X}Z +
  (I - S_{t})\barr{\del}_{M}P\barr{\del}_{X}Z +
  Q\barr{\del}_{M}\barr{\del}_{X}Z + (\barr{\del}_{X}- \barr{\del}_{M})F.
\end{equation}
Using the homotopy formula, $\barr{\del}_{M}P\barr{\del}_{X}Z =
  \barr{\del}_{X}Z - Q\barr{\del}_{M}\barr{\del}_{X}Z$,
a second time, in the second term, gives
\begin{equation}
  \barr{\del}_{X}Z_{*} = I_{1} + I_{2} +I_{3} +I_{4} ,
\end{equation}
where
\begin{equation}
\begin{array}{lclclcl}
  I_{1} & = & (I-S_{t})\barr{\del}_{X}Z & ; & I_{2} & =
        & (\barr{\del}_{M}- \barr{\del}_{X})S_{t}P\barr{\del}_{X}Z ; \\
  I_{3} & = & [S_{t},\barr{\del}_{M}]P\barr{\del}_{X}Z & ; &
  I_{4} & = & S_{t}Q(\barr{\del}_{M}- \barr{\del}_{X})\barr{\del}_{X}Z .
\end{array}
\end{equation}
One may see that formally each term contains the product of 2
\lq\lq small\rq\rq\ factors, hence should be \lq\lq smaller\rq\rq.
The precise estimates will be given below.

%444444444444444444444444444444444444444444444444444444444444444444444444444444
\section{Renormalization of the embedding.}
\setcounter{equation}{0}
While the change of section 3 will take an embedding with \lq\lq small\rq\rq\
error to one with \lq\lq smaller\rq\rq\ error, it tends to destroy the
normalizations of lemma 2.1, which we use to set up and estimate the homotopy
formula. In this section we give a precise procedure to
restore these normalizations, while retaining a \lq\lq smaller\rq\rq\ error.
We point out that this section represents a considerable simplification
and improvement over the corresponding argument in [24], where third
order normalization was used.

We now assume that $\|\barr{\del}_{X}Z\|$ is
\lq\lq\ small\rq\rq, and (see (1.8),(1.9))
\begin{equation}
  \barr{\del}_{X}Z(0) = 0, \; \Leftrightarrow \;
  A_{\gra}^{\;\barr{\grb}}(0) = 0, \; B_{\gra}(0) = 0,
\end{equation}
which clearly holds for the initial embedding of section one.  We shall
renormalize $Z_{*}=Z+F$, $F$ given by (3.2),  using the first order
Taylor polynomial of $F$ at $0$,
\begin{equation}
\begin{array}{rcl}
  F(z',x^{n}) & = & K_{0}+K_{\gra}z^{\gra} + K_{\barr{\gra}}z^{\barr{\gra}}
  + K_{n}x^{n} + F_{(2)}, \\
   & \equiv &  K_{0}+\hat{K}\cdot x + F_{(2)}(x) , \\
   F_{(2)}(z',x^{n}) & = & O(2),
\end{array}
\end{equation}
where each $K_{j}=(K_{j}^{\gra},K_{j}^{n})\in\mathbf{C}^{n}$ is a
constant
vector, and  \lq\lq small\rq\rq .  Notice that
$X_{\barr{\gra}}Z_{*}(0) = X_{\barr{\gra}}F(0) = K_{\barr{\gra}}$, so that
$ |K_{\barr{\gra}}| \leq \|\barr{\del}_{X}Z_{*}\|_{0}$
is actually a \lq\lq smaller\rq\rq\ quantity.

Now we replace $Z_{*} = Z + F$ by
\begin{eqnarray}
  Z_{*} & = & Z + F + E = Z + (F_{(2)} - iK_{n}y^{n})    , \\
  E     & = & -K_{0} - K_{\gra}z^{\gra} - K_{n}z^{n} -
               K_{\barr{\gra}}z^{\barr{\gra}}.
\end{eqnarray}
It follows that $\barr{\del}_{X}Z_{*}(0) = 0$, since $y^{n}=O(2)$; and
$\|\barr{\del}_{X}Z_{*}\|$ remains \lq\lq smaller\rq\rq, since each
$X_{\barr{\grb}}E = -K_{\gra}X_{\barr{\grb}}z^{\gra}
-K_{n}X_{\barr{\grb}}z^{n} - K_{\barr{\grb}}$ is \lq\lq smaller\rq\rq .
We write (4.3) as
\begin{equation}
\begin{array}{lcl}
  z_{*}^{\gra} & = & z^{\gra} + f_{(2)}^{\gra}, \; \;
  f_{(2)}^{\gra} = F_{(2)}^{\gra} - iK_{n}^{\gra}y^{n}, \\
  z_{*}^{n} & = & z^{n} + F_{(2)}^{n} - iK_{n}^{n}y^{n}.
\end{array}
\end{equation}
With $K^{n}_{n}= K^{'n}_{n}+iK^{''n}_{n}$, the second equation becomes
\begin{equation}
\begin{array}{lcl}
  x_{*}^{n} & = & x^{n} + f_{(2)}^{n}, \; \;
  f_{(2)}^{n} = \mbox{Re}(F_{(2)}^{n}) + K^{''n}_{n}y^{n} , \\
  y_{*}^{n} & = & y^{n} + \mbox{Im}(F_{(2)}^{n}) - K^{'n}_{n}y^{n} \\
            & = & |z'|^{2} + h + \hat {h} , \\
  \hat{h}   & = & \mbox{Im}(F^{n}_{(2)}) - K^{'n}_{n}y^{n}.
\end{array}
\end{equation}

The first equations in (4.5) and (4.6) define a map $f$, which is
the projection of the map $Z+F+E$ to the real hyperplane
$y^{n}=0$. It is a local diffeomorphism  of neighborhoods of $0$
in $\mathbf{R}^{2n-1}$.  We let $g$ be its inverse.  Then
\begin{equation}
  f = I + f_{(2)}, \; \; \; f^{-1} = g = I + g_{(2)},
\end{equation}
where $f_{(2)}=O(2)$, $g_{(2)}=O(2)$.
We define the new normalized embedding $Z_{1}$  and new real hypersurface
$M_{1}$, the image of $Z_{1}$,  by
\begin{equation}
\begin{array}{rcl}
  Z_{1} & = & Z_{*}\circ g = (Z+F+E)\circ g = (z',z^{n}_{1}), \\
  z^{n}_{1} & = & x^{n}+iy^{n}_{1}(z',x^{n}), \\
  M_{1} & = & \{r_{1} = 0\}, \; \; \;  r_{1} = -y^{n}+
      |z'|^{2} + h_{1}(z',x^{n}).
\end{array}
\end{equation}
We have
\begin{equation}
  h_{1}(z',x^{n}) = (h + \hat{h})\circ g + \grd_{\gra\barr{\grb}}
   (z^{\gra}g^{\barr{\grb}}_{(2)}+g^{\gra}_{(2)}z^{\barr{\grb}}+
      g^{\gra}_{(2)}g^{\barr{\grb}}_{(2)}),
\end{equation}
or
\begin{equation}
  h_{1} - h = (h\circ g - h) + \hat{h}\circ g + \grd_{\gra\barr{\grb}}
   (z^{\gra}g^{\barr{\grb}}_{(2)}+g^{\gra}_{(2)}z^{\barr{\grb}}+
      g^{\gra}_{(2)}g^{\barr{\grb}}_{(2)}).
\end{equation}

To determine the new adapted frame field $X^{1}_{\gra}$ for $Z_{1}$,
we think of $f$ as mapping our given CR structure to an equivalent
one.  To achieve $\grd_{\gra}^{\;\grb}=X^{1}_{\gra}z^{\grb}_{1}$, we
define a matrix of functions $C^{\;\grb}_{\gra}X_{\grb}$ by the
following.  (Note that $f_{*}(V)_{x}[h]=V_{g(x)}[h\circ f]$, and
$g\circ f = id$.)
\begin{equation}
\begin{array}{rcl}
  X^{1}_{\gra} & = & f_{*}(C^{\;\grb}_{\gra}X_{\grb}), \\
  \grd^{\;\grb}_{\gra} & = & C^{\;\grg}_{\gra}X_{\grg}[Z+F+E]^{\grb} =
    C^{\;\grg}_{\gra}(\grd^{\;\grb}_{\grg} + X_{\grg}f^{\grb}_{(2)}), \\
    C^{\;\grb}_{\gra} & = & \grd^{\;\grb}_{\gra} -X_{\gra}f^{\grg}_{(2)}
    (\grd^{\;\grb}_{\grg} + X_{\grg}f^{\grb}_{(2)})^{-1},
\end{array}
\end{equation}
where $(\cdot)^{-1}$ indicates matrix inverse. This relies on the smallness
of $X_{\gra}f^{\grb}_{(2)}$ near $0$.   We also have, using
$g_{*}f_{*}=I$,
\begin{equation}
\begin{array}{rcl}
  \barr{\del}_{X^{1}}Z_{1} & \equiv & \{X^{1}_{\barr{\gra}}Z_{1}\}
       =  \{(g_{*}(X^{1}_{\barr{\gra}})[Z+F+E])\circ g\} \\
      & = & \{(C^{\;\barr{\grb}}_{\barr{\gra}}W_{\barr{\grb}})\circ      g\}, \\
  W_{\barr{\grb}} & = & X_{\barr{\grb}}[Z+F] - X_{\barr{\grb}}[Z+F](0) +
     \\ &  &
     -\del_{\gra}F(0)X_{\barr{\grb}}z^{\gra}
        -\del_{x^{n}}F(0)X_{\barr{\grb}}z^{n} .
\end{array}
\end{equation}

So far we have considered our maps $Z$ and $Z_{1}$ on the level of germs.
As such, we have two well defined functional relations,
\begin{equation}
  Z_{1} = \mathcal{H}(Z,F), \; \; \;
  \barr{\del}_{X^{1}}Z_{1} = \mathcal{G}(Z,F,\barr{\del}_{X}Z,DF).
\end{equation}
Here the operator  $\mathcal{H}$ is zero-th order, but involves composition
with the inverse of the projected map $f$.  The operator $\mathcal{G}$ is a
first order differential operator also involving the composition.

To compare the actual domains of definition of $Z$ and $Z_{1}$, we must
consider the inverse mapping theorem more carefully. We use the
notation $\|\cdot\|_{\grr,k} = \|\cdot\|_{C^{k}(D_{\grr})}$ of the
next section.
\begin{lemma}
   Let $D$, $h$, $\grg_{0}$, $\grr_{0}$, $D_{\grr_{0}}$, and
   $0<\grs<1$  be as in lemma 2.1, and let $\grr\leq\grr_{0}$.
   Suppose that $f=I+f_{(2)}$ is a smooth map from
   $D_{\grr}$ to $\mathbf{R}^{2n-1}$, with
\begin{equation}
  f_{(2)}=O(|x|^{2}), \; \;  \|f_{(2)}\|_{\grr_,1} \, \leq \, \grs/5.
\end{equation}
Then $f$ maps a compact smooth neighborhood
$U_{\grr}\subseteq D_{\grr}$ of $0$, diffeomorphically onto
$D_{\grr(1-\grs)}$.  The inverse map $g=f^{-1}$ has
the form $g=I + g_{(2)}$, where
\begin{equation}
     g_{(2)}=O(|x|^{2}), \; \;
     \|g_{(2)}\|_{\grr(1-\grs),a} \, \leq \, c_{a} \|f_{(2)}\|_{\grr,a},
\end{equation}
for $0\leq a\leq2$, where the constant $c_{a}$ is independent of the
domain and map.
\end{lemma}
\begin{proof}
The proof follows a standard argument for the inverse function theorem.
We fix $x_{*}\in D_{\grr(1-\grs)}$ and apply the contraction
mapping principle to $x\mapsto w(x) = x_{*}-f_{(2)}(x)$, on the domain
$D_{\grr}$.  Suppose $\|f_{(2)}\|_{\grr,1}\leq\grg\grs$, with $\grg>0$
to be chosen. By the convexity of $D_{\grr}$ and the mean value theorem
on the segment $[0,x]$, we get $|f_{(2)}(x)|\leq\|f_{(2)}\|_{\grr,1}|x|$.
Then $|w(x)-x_{*}|\leq \grg\grs|x|\leq\grg\grs\sqrt{2}\grr$, by (2.3),
for $x\in D_{\grr}$.  By (2.6) this is bounded by
$\sqrt{2}\grg\hat{c}$dist$(\del D_{\grr_{0}(1-\grs)},\del D_{\grr_{0}})$.
But $\sqrt{2}\hat{c}\leq2(2+\grg_{0})\leq5$
Thus $w$ maps $D_{\grr}$ into itself, if $\grg=1/5$.
Similarly for $x_{0}$, $x_{1}$ in  $D_{\grr}$,
$|w(x_{1})-w(x_{0})|\leq \grg\grs|x_{1}-x_{0}|$,
so $w$ is a contraction there. It has a unique fixed point $x=g(x_{*})$,
$f(g(x_{*})) = x_{*}$.  The map $g$ is just as smooth as $f$, by the
standard argument.

Since $f(0)=0$, $df(0)=I$, we have $g(0)=0$, and by the chain rule,
$dg(0)=I$.  Thus $g=I+g_{(2)}$, $g_{(2)}=O(2)$, and the jacobian matrices
further satisfy
\begin{equation}
dg_{(2)}(x) \, = \, \grPh\circ df_{(2)}\circ g(x), \;
\grPh(W) = -W(I+W)^{-1} .
\end{equation}
Since $g_{(2)}(x)$ vanishes for $x=0$, we can bound it in terms of
its first derivatives, as above.  By (4.14)
$|df_{(2)}|\leq1/2$, so (4.16) gives (4.15) for $a=1$.

For the H\"older continuity, let $0<\gra<1$.  We estimate the two
compositions $u=\grPh\circ df_{(2)}$, $v=u\circ g$ differently,
\begin{eqnarray}
  |v(x_{1})-v(x_{0})| & \leq & H_{\gra}(u)|g(x_{1})-g(x_{0})|^{\gra}
 \leq H_{\gra}(u)(1+\|g_{(2)}\|_{1})^{\gra}|x_{1}-x_{0}|^{\gra} ,
 \nonumber \\
 |u(x_{1})-u(x_{0})| & \leq & \|\grPh\|_{1}|df_{(2)}(x_{1})-df_{(2)}(x_{0})|
 \leq \|\grPh\|_{1}H_{\gra}(df_{(2)})|x_{1}-x_{0}|^{\gra}. \nonumber
\end{eqnarray}
Combining gives (4.15) for $0<a<2$.

Finally, taking first partial derivatives in (4.16) and estimating
gives (4.15) for $a=2$, in view of (4.14).
\end{proof}

Next we assume that $f$ (4.7) is the projection of $Z+F+E$,
and that $h_{1}$ is given by (4.9).  We want to compare two of the
domains $D_{\grr_{1}}(h_{1})$ and $D_{\grr}(h)$.
\begin{lemma}
 There is a constant $\grg_{1}>0$ such that the following holds.
 Let $D$, $h$, $\grg_{0}$, $\grr_{0}$, $D_{\grr_{0}}$, and $\grs$  be
 as in lemma 4.1, with $0<\grr_{0}\leq1$ and $0<\grs<1/2$. Let
$\grr\leq\grr_{0}$, and $f$ and $h_{1}$ be as above.   If the map $F$
satisfies
\begin{equation}
    \|F\|_{\grr,1}\leq\grg_{1}\grr\grs ,
\end{equation}
then lemma 4.1 holds, and the functions  $g$, $Z_{1}$,  and $h_{1}$
are defined on a  neighborhood of
$D_{\grr(1-\grs)}\equiv D_{\grr(1-\grs)}(h)$, and
\begin{equation}
D_{\grr(1-2\grs)}(h_{1}) \subseteq D_{\grr(1-\grs)}(h).
\end{equation}
\end{lemma}
\begin{proof}
For the proof note that (4.5), (4.6) give
\begin{eqnarray}
  \|f_{(2)}\|_{\grr,1} & \leq & \|F_{(2)}\|_{\grr,1} +
   |dF(0)|\|y^{n}\|_{\grr,1}   \nonumber \\
    & \leq & 3\|F\|_{\grr,1} + |dF(0)|(3\sqrt{2}\grr
              + \|h\|_{\grr,1})  \\
    & \leq & \|F\|_{\grr,1}(3+3\sqrt{2}+1/2) \leq
    8\|F\|_{\grr,1}. \nonumber
\end{eqnarray}
Thus, if we choose $\grg_{1}<1/45$, then (4.14) and hence lemma 4.1
hold. We apply (4.10) and the Schwarz inequality to get
\begin{equation}
\begin{array}{c}
  \grps(x) - \grps_{1}(x) = h(x) - h_{1}(x)  \leq   \\
 |h(g(x)) - h(x)| + |\hat{h}(g(x))| + (2|x| + |g_{(2)}(x)|)|g_{(2)}(x)|.
\end{array}
\end{equation}

We shall bound the first term on the right in (4.20) by the third
term, as in the proof of lemma 2.1.
For $x\in D_{\grr(1-\grs)}(h)$, the segment
$[x,g(x)]$ lies in $D_{\grr}(h)$.  Thus we have
\begin{eqnarray}
   g_{t}(x) & = & x + tg_{(2)}(x) \mbox{ (def.)} ,\\
  h(g(x)) - h(x) & = & \int_{0}^{1}\del_{j}h(g_{t}(x))g^{j}_{(2)}(x)dt , \\
  | \del_{j}h(g_{t}(x)) | & \leq & \|h\|_{\grr,2}(|x|+|g_{(2)}(x)|) , \\
 |h(g(x)) - h(x)| & \leq & \grg_{0}(|x|+|g_{(2)}(x)|)|g_{(2)}(x)| .
\end{eqnarray}

To estimate the first and third terms on the right of (4.20), we note that
$2|x| + |g_{(2)}(x)| \leq 3|x| + |g(x)| \leq \sqrt{2}(3\grr +\grr)
\leq 4\sqrt{2}$. Also, by (4.15), (4.19)
\begin{equation}
  |g_{(2)}(x)|\leq\|g_{(2)}\|_{\grr(1-\grs),1}|x|
  \leq 9c_{1}\|F\|_{\grr,1}|x|
  \leq 9c_{1}\sqrt{2}\grr\|F\|_{\grr,1}.
\end{equation}

For the second term, note that for $x\in D_{\grr}(h)$, (4.6) gives
\begin{eqnarray}
  |\hat{h}(x)| & \leq & \|F_{(2)}\|_{\grr,1}|x| +
     |dF(0)|(|z'|^{2} + |h(z',x^{n})|) \\  \nonumber
  & \leq & \|F\|_{\grr,1}(2 + \sqrt{2}\grr +
  \|h\|_{\grr,1})|x| \\
  & \leq & 5\|F\|_{\grr,1}|x|. \nonumber
\end{eqnarray}
For $x\in D_{\grr(1-\grs)}(h)$, we have $g(x)\in D_{\grr}(h)$,
so this gives
\begin{equation}
  |\hat{h}(g(x))| \; \; \leq \; \;  5\sqrt{2}\grr\|F\|_{\grr,1}.
\end{equation}

Combining the three terms and using (4.17) gives,
for $x\in D_{\grr(1-\grs)}(h)$,
\begin{equation}
  \grps(x) - \grps_{1}(x) \; \leq \;
  \tilde{c}\grr\|F\|_{\grr,1}
  \; \leq \;  \tilde{c}\grg_{1}\grr^{2}\grs,
\end{equation}
for an absolute constant $\tilde{c}$.

We want to show that $\grps(x)\leq\grr^{2}(1-\grs)^{2}$, if
$\grps_{1}(x)\leq\grr^{2}(1-2\grs)^{2}$.  This comes down to
$3\grs + \tilde{c}\grg_{1}\leq 2$.  With $\grs<1/2$, we require
$\tilde{c}\grg_{1}<1/2$. This proves the lemma.
\end{proof}

By combining the results of sections 3 and 4, we get the main
step, $Z\mapsto Z_{1}$, in our iteration procedure.  Basically, we
shall need to control the new approximate embedding $Z_{1}$, or
equivalently $h_{1}$, and the new error $\barr{\del}_{X_{1}}Z_{1}$.
As lemma 4.2 indicates, the control of $F$ (3.2) will be central to
this process.  This is begun in the next two sections.

%555555555555555555555555555555555555555555555555555555555555555555555555555555

\section{H\"older estimates.}
\setcounter{equation}{0}
First we recall some basic results for standard H\"older norms
$\|u\|_{a} = \|u\|_{C^{a}(D_{\grr})}$, $0\leq a < \infty$, for
functions $u$ on a smooth, bounded, strictly convex domain $D_{\grr}$,
with in-radius comparable to $\grr$,
\begin{equation}
  B(\gre_{1}\grr)\subseteq D_{\grr} \subseteq B(\gre_{2}\grr)\subset
  \mathbf{R}^{n}, \; 0 < \gre_{1} < \gre_{2} .
\end{equation}
For reference, see the appendices of [10] or [6], which we follow with
some adaptation. Next we take $D_{\grr}= D_{\grr}(h)$, setting
$\|u\|_{\grr,a} = \|u\|_{a}$, and make estimates for the new function
$h_{1}$.  From now on we shall assume $0 <\grr \leq 1$.

Perhaps most basic is the interpolation estimate, where $a < c < b$,
$0 < \grl < 1$, $c = \grl a + (1-\grl) b$,
\begin{equation}
   \|u\|_{c} \; \leq \; c_{a}\grr^{-c}\|u\|_{a}^{\grl}\|u\|_{b}^{1-\grl}.
\end{equation}
From this one derives the convexity estimate, for functions $u$, $v$
on perhaps different domains $D_{\grr}$, $D_{\grt}$ ,
\begin{equation}
   \|u\|_{a}\|v\|_{b} \; \leq \; c_{a,b}\grr^{-a}\grt^{-b}
  (\|u\|_{a_{1}}\|v\|_{b_{1}} + \|u\|_{a_{2}}\|v\|_{b_{2}} ),
\end{equation}
where $(a,b)=\grl(a_{1},b_{1}) + (1-\grl)(a_{2},b_{2})$, $0<\grl<1$.

The product-rule estimate for $u,v$ on $D_{\grr}$ is
\begin{equation}
  \|uv\|_{a} \; \leq \; c_{a}\grr^{-a} (\|u\|_{a}\|v\|_{0} +
  \|u\|_{0}\|v\|_{a}).
\end{equation}
These estimates can be derived as in [10], with constants independent
of $\grr$, but using scale invariant norms, and then passing back to
standard norms. (To get \lq\lq scale invariant \rq\rq\, norms, multiply
the sup norm of the $j$-th order derivatives by $\grr^{j}$, $0< j\leq k$,
and the $\gra$-H\"older ratio of the $k$-th order derivatives by
$\grr^{k+\gra}$, and add together.)

As indicated in [10], there are two ways to make a chain-rule estimate.
Our first is somewhat weaker than the corresponding one in [10], but
in a more precise form.  For maps $g:D_{\grr}\rightarrow D_{\grt}$ and
$u:D_{\grt}\rightarrow \mathbf{R}^{N}$, $0<\grr,\grt\leq 1$, we have
\begin{eqnarray}
  \|u\circ g\|_{a} & \leq & K_{a}(\|u\|_{a} + \|u\|_{1}\|g\|_{a}), \\
  K_{a} & = & c_{a}\grt^{-2a}\grr^{-a^{2}}(1+\|g\|_{1})^{2a}. \nonumber
\end{eqnarray}
In case $D_{\grt}= B(1/2)$ and $u$ is a fixed rational or analytic
function on $B(1)$, (specifically the matrix inverse $u(W) =
(I+W)^{-1}$), we have
\begin{eqnarray}
  \|u\circ g\|_{a} & \leq &  K_{a}\|g\|_{a} , \\
 K_{a} & = & c_{a}\grr^{-a}(1+\|g\|_{0})^{a-1}. \nonumber
\end{eqnarray}

We augment lemma 4.1 with the following. It is a refinement of the
inverse mapping estimate of [10].  See appendix A of [6].
\begin{lemma}
  With the hypothesis of lemma 4.1, we also have
\begin{equation}
  \|g_{(2)}\|_{\grr(1-\grs),a} \; \leq \;
                    c_{a}\grr^{-4(a+2)}\|f_{(2)}\|_{\grr,a},
\end{equation}
for $0\leq a < \infty$. For $0\leq a \leq 2$ there is no $\grr$-factor,
by (4.15).
\end{lemma}

In lemma 4.2, $f_{(2)}$ is given by (4.5), (4.6) and (4.2), (1.5), and
we assume $\|h\|_{\grr,2}\leq\grg_{0}<1/2$, and
$\|F\|_{\grr,1}\leq \grg_{1}\grr\grs <1/2$.  Then as in (4.19) we have
\begin{equation}
  \|f_{(2)}\|_{\grr,a} \; \leq \; c_{a}(\|F\|_{\grr,a} +
   \|F\|_{\grr,1}\|h\|_{\grr,a}), \; \; 1\leq a <\infty.
\end{equation}

\begin{lemma}
    Let the hypotheses be as in lemma
  4.2.  For $\grr_{1}=\grr(1-\grs)$, we have
  (the domains are $D_{\grr_{1}}(h)$ and $D_{\grr}(h)$)
\begin{eqnarray}
  \|h_{1}\|_{\grr_{1},a} & \leq &
       c_{a}\grr^{-a_{2}}(\|h\|_{\grr,a} + \|F\|_{\grr,a}),\\
  \|h_{1}-h\|_{\grr_{1},a} & \leq & c_{a}\grr^{-a_{2}}
    (\|F\|_{\grr,1}\|h\|_{\grr,a+1}+ \|F\|_{\grr,a}),
\end{eqnarray}
for $1\leq a <\infty$, where $a_{2}$ is quadratic in $a$.
\end{lemma}

\begin{proof}
To prove (5.9), we use the chain-rule and product-rule estimates in (4.9)
to get
\begin{eqnarray}
  \|h_{1}\|_{\grr_{1},a} & \leq & c_{a}\grr^{-a_{2}} ( \|h+\hat{h}\|_{\grr,a} +
  \|h+\hat{h}\|_{\grr,1} \|g_{(2)}\|_{\grr_{1},a} ) + \\ &  &  \nonumber
  + (2 +  c_{a}\grr^{-a}\|g_{(2)}\|_{\grr_{1},0})\|g_{(2)}\|_{\grr_{1},a} .
\end{eqnarray}
From (4.6) and (1.4), (1.5) we get
\begin{equation}
  \|\hat{h}\|_{\grr,a} \; \leq \;
   c_{a}(\|F\|_{\grr,a} + \|F\|_{\grr,1}\|h\|_{\grr,a}).
\end{equation}
Now we apply (5.7), (5.8) and use $\|F\|_{\grr,1}\leq 1$.  After
simplifying, we get (5.9) (with $a_{2}= a^{2}+7a+8$).

To prove (5.10) we estimate the first term in (4.10) using calculus,
as in (4.21)-(4.22), and the product rule estimate, getting (with
norms over the appropriate domains)
\begin{eqnarray}
 \|h\circ g - h\|_{a} & \leq & c_{a}\grr^{-a} \int_{0}^{1}
   \{ \|Dh\circ g_{t}\|_{a}\|g_{(2)}\|_{0} +
        \|Dh\circ g_{t}\|_{0} \|g_{(2)}\|_{a} \}  \\
       & \leq & c_{a}\grr^{-a-a^{2}} \{\|h\|_{a+1}\|g_{(2)}\|_{0} +
               \|h\|_{2}\|g_{(2)}\|_{a} \}, \nonumber
\end{eqnarray}
and we use $\|h\|_{2}\leq1$.  The other terms are treated as in the
proof of (5.9), with $h$ dropped from the expression $(h+\hat{h})$.
After simplifying, we get (5.10).
\end{proof}

%666666666666666666666666666666666666666666666666666666666666666666666666666666
\section{A general estimate for the new error.}
\setcounter{equation}{0}
We derive a \lq\lq general\rq\rq\ estimate for the new error
$\barr{\del}_{X^{1}}Z_{1}$ in terms of the previous embedding and the function
$F$.  This will be specialized to give coarse and fine estimates for the
new error. We assume that our vector fields $X_{\gra}$ are in the
H\"older class $C^{m}$, $m\in\mathbf{R}$, $m\geq2$.

We first estimate the $X_{\gra}$ derivative of a function $u$. This is
straight forward using the product-rule estimate twice, and (1.8), (1.9).
\begin{lemma}
  For $0\leq a \leq m$, we have, on a fixed domain,
\begin{equation}
\begin{array}{rcl}
  \|X_{\barr{\gra}}u\|_{a} & \leq & K'_{a}(\|u\|_{a+1} +
   (\|\barr{\del}_{X}Z\|_{a} + \|h\|_{a+1}) \|u\|_{1}), \\
  \|X_{\barr{\gra}}u - Y_{\barr{\gra}}u\|_{a} & \leq &
      K'_{a}(\|\barr{\del}_{X}Z\|_{0}\|u\|_{a+1} +
      (\|\barr{\del}_{X}Z\|_{0}\|h\|_{a+1} + \|\barr{\del}_{X}Z\|_{a})
          \|u\|_{1}), \\
    K'_{a} & = & c_{a}\grr^{-2a}(1+\|\barr{\del}_{X}Z\|_{0}).
\end{array}
\end{equation}
\end{lemma}
The first estimate of the lemma is applied to get
\begin{eqnarray}
  \|X_{\gra}f^{\grb}_{(2)}\|_{\grr,0} & \leq &
  c_{0}(1+ \|\barr{\del}_{X}Z\|_{\grr,0})^{2}\|f_{(2)}\|_{\grr,1}  \\
  & \leq & 3c_{0}\|f_{(2)}\|_{\grr,1}, \nonumber
\end{eqnarray}
if we assume $\|\barr{\del}_{X}Z\|_{\grr,0}\leq 1/2$, say.  Then, if
$\|f_{(2)}\|_{\grr,1}$ is sufficiently small, the inverse matrix in
(4.11) will exist and (4.12) will be valid.

We have the following general estimate.
\begin{lemma}
 Let the hypotheses of lemma 4.2 hold,
 $\|\barr{\del}_{X}Z\|_{\grr,1}\leq 1/2$, and $\grr_{1}=\grr(1-\grs)$.
Then
\begin{equation}
\begin{array}{rcl}
  \|\barr{\del}_{X^{1}}Z_{1}\|_{\grr_{1},a} & \leq & c_{a}\grr^{-a_{2}} \{
  \|\barr{\del}_{X}(Z+F)\|_{\grr,a} +
       \|F\|_{\grr,1}\|\barr{\del}_{X}Z\|_{\grr,a} + \\
  &  & +(\|\barr{\del}_{X}Z\|_{\grr,0} +\|F\|_{\grr,1})(\|F\|_{\grr,a+1} +
      \|F\|_{\grr,1}\|h\|_{\grr,a+1}) + \\
  &  & + \|\barr{\del}_{X}Z\|_{\grr,1}(\|F\|_{\grr,a} +
      \|F\|_{\grr,1}\|h\|_{\grr,a}) \},
\end{array}
\end{equation}
for $1\leq a \leq m$.
\end{lemma}
\begin{proof}
We must estimate (4.12), which we abbreviate as
\begin{equation}
  \barr{\del}_{X^{1}}Z_{1} \; = \; (\barr{C}\barr{W})\circ g, \; \; \;
  C \; = \; I - Xf_{(2)}(I + Xf_{(2)})^{-1}.
\end{equation}
For the following 4 estimates, we use, respectively, chain rule (5.5)
and lemma 5.1; the product rule and chain rule (5.6); the product
rule; and lemma 6.1.
\begin{eqnarray}
  \|\barr{\del}_{X^{1}}Z_{1}\|_{\grr_{1},a} & \leq &  c_{a}\grr^{-a_{2}}
   (\|CW\|_{a} + \|CW\|_{1}\|f_{(2)}\|_{a}); \\
  \|C\|_{a} & \leq & c_{a}\grr^{-2a}(1+\|Xf_{(2)}\|_{a}); \\
  \|CW\|_{a} & \leq & c_{a}\grr^{-3a} (\|W\|_{a} + \|W\|_{0}\|Xf_{(2)}\|_{a}),\\
  \|Xf_{(2)}\|_{a} & \leq & c_{a}\grr^{-2a}(\|f_{(2)}\|_{a+1} +
   \|f_{(2)}\|_{1}( \|\barr{\del}_{X}Z\|_{a} + \|h\|_{a+1})).
\end{eqnarray}
In combining and simplifying, we note that under our hypotheses,
$1+\|f_{(2)}\|_{1}\leq2$,
$\|\barr{\del}_{X}Z\|_{1} + \|h\|_{2}\leq2$, and
\begin{equation}
  \|f_{(2)}\|_{2}\|f_{(2)}\|_{a}\; \leq \;
  2c_{a}\grr^{-2-a}\|f_{(2)}\|_{1}\|f_{(2)}\|_{a+1},
\end{equation}
by the convexity estimate (5.3).    This leads to
\begin{eqnarray}
  \|\barr{\del}_{X^{1}}Z_{1}\|_{\grr_{1},a} & \leq &  c_{a}\grr^{-a_{2}}
 \{ \|W\|_{a} + \|W\|_{1}\|f_{(2)}\|_{a} + \\ &  &  \nonumber
  + \|W\|_{0} (\|f_{(2)}\|_{a+1} + \|f_{(2)}\|_{1}
  (\|\barr{\del}_{X}Z\|_{a} + \|h\|_{a+1}) ) \} .
\end{eqnarray}
From (4.12) we get
\begin{equation}
  \|W\|_{a} \; \leq \; 2\|\barr{\del}_{X}(Z+F)\|_{a} +
         \|F\|_{1}\|\barr{\del}_{X}Z\|_{a}.
\end{equation}
We substitute this for the $W$-terms and use the convexity
estimate (5.3) for $\|F\|_{2}\|F\|_{a}$ and $\|F\|_{2}\|h\|_{a}$
as above.  This gives the lemma.\end{proof}

%777777777777777777777777777777777777777777777777777777777777777777777777777777
\section{Estimates for $S_{t}$, $P$, and $Q$. }
\setcounter{equation}{0}
Here we summarize some results on a standard smoothing operator $S_{t}$,
and give an estimate for the Friedrichs commutator.  Then we recall some
necessary results from [6] estimating the homotopy operators $P$ and $Q$.

By lemma 2.1  we may define a standard smoothing operator
$S_{t}:C(D_{\grr}) \rightarrow C^{\infty}(D_{\grr_{1}})$,
$\grr_{1}=\grr(1-\grs)$, $ 0 < t < \hat{c}^{-1} \grr\grs$,
$\hat{c}=3\sqrt{2}$,
\begin{equation}
  S_{t}u(x) = \int_{|x-y|<\grr\grs}u(y)\grch_{t}(x-y)dy =
   \int_{|z|<1}\grch(z)u(x-tz)dz,
\end{equation}
where $spt\grch\subset\subset\{|z|<1\}$, $\int \grch(z)dz =1$, and
$\int z^{I}\grch(z)dz = 0$ for $0<|I|< 2m$, say.

We have the basic estimates [17], [24]
\begin{equation}
\begin{array}{rclcl}
  \|S_{t}u\|_{\grr_{1},a} & \leq & c_{a}\grr^{-a}t^{b-a}\|u\|_{\grr,b} & , &
   0\leq b\leq a < \infty, \\
  \|(I-S_{t})u\|_{\grr_{1},a} & \leq & c_{m}\grr^{-a}t^{b-a}\|u\|_{\grr,b} & , &
   0\leq a \leq b < 2m . \\
\end{array}
\end{equation}
For $b-a\in\mathbf{Z}$ these hold without the $\grr$-factors.  For
$b-a\in\mathbf{R}$ they follow by means of the interpolation
estimate (5.2).

The commutator $[S_{t},\barr{\del}_{M}]$ is equivalent to all
$[S_{t},Y_{\barr{\gra}}]= [S_{t},w]\del_{x^{n}}$, where
$w = -r_{\barr{\gra}}/2r_{\barr{n}}$.
\begin{lemma}
For $k\in\mathbf{Z}^{+}$, $0\leq\gra\leq 1$,
$0< t < \hat{c}^{-1}\grr\grs$, we have
\begin{eqnarray}
  \|[S_{t},\barr{\del}_{M}]u\|_{\grr_{1},k} &  \leq &
     c_{m}\grr^{-k-\gra} t^{\gra}
    \{\|r\|_{\grr,2}\|u\|_{\grr,k+\gra} +
   \|r\|_{\grr,k+2}\|u\|_{\grr,\gra}\} , \\
\|[S_{t},\barr{\del}_{M}]u\|_{\grr_{1},k+\gra} &  \leq &
     c_{m}\grr^{-k-\gra}
    \{\|r\|_{\grr,2}\|u\|_{\grr,k+\gra} +
   \|r\|_{\grr,k+2+\gra}\|u\|_{\grr,0}\}.
\end{eqnarray}
\end{lemma}
\begin{proof}
  We set $v(x) =[S_{t},w]\del_{x^{n}} u(x)$, and integrate by parts to
take the derivative $\del_{x^{n}}$ off $u$.  This gives
\begin{equation}
\begin{array}{rcl}
   v(x) & = & \int_{|z|\leq 1}R(x,z,t)(u(x-tz)-u(x))dz, \\
  R(x,z,t) & = & \del_{z^{n}}\grch(z)t^{-1}(w(x-tz) - w(x)) -
          \grch(z)\del_{x^{n}}w(x-tz) .
\end{array}
\end{equation}
For the sup norm and Holder ratio, we readily derive
\begin{equation}
\begin{array}{rcl}
  \|v\|_{0} & \leq & c_{0}\|w_{x}\|_{0}t^{\gra}H_{\gra}(u) , \\
  H_{\gra}(v) & \leq & c_{\gra}(\|w_{x}\|_{0}H_{\gra}(u) +
                  H_{\gra}(w_{x})\|u\|_{0}).
\end{array}
\end{equation}
We take a first order $x$-derivative $D$ by means of
\begin{equation}
  D[S_{t},w]\del_{x^{n}}u \, = \, [S_{t},Dw]\del_{x^{n}}u(x) +
  [S_{t},w]\del_{x^{n}}Du(x).
\end{equation}
We take $k$ such derivatives and apply (7.6) and the convexity
estimate (5.3).  This gives the lemma.
\end{proof}

In [6] we have derived the following estimates, coarse and fine, for the
homotopy formula.
\begin{lemma} Let the domain $D_{\grr}$ and the corresponding hypersurface
$M_{\grr}$ be as in lemma 2.1.  Then the operators $P$ and $Q$ in the
homotopy formula (2.13) for $M_{\grr}$ satisfy the following estimates.
\begin{equation}
  \|P\grph\|_{\grr_{1},a} \; \leq \; K(a)
     \{\|\grph\|_{\grr,a} + \|h\|_{\grr,a+2}\|\grph\|_{\grr,0}\},
      \; \; K(a)=c_{a}(\grr\grs)^{-s(a)} ,
\end{equation}
where $\grr_{1} = \grr(1-\grs)$, and
$s(a)$ is some polynomial in $a$,  $0\leq a < \infty$.  Also,
for $k\in\mathbf{Z}$, $0\leq k$, $\gra=1/2$,  and
$\grb= 5/2$,
\begin{equation}
   \|P\grph\|_{\grr_{1},k+\gra} \; \leq \; K(k+\gra)
       \{(1+\|h\|_{\grr,\grb})\|\grph\|_{\grr,k} +
   \|h\|_{\grr,k+\grb}\|\grph\|_{\grr,0}\}.
\end{equation}
For the special case of forms $\grph$ of type (0,1), $k=0$, and
$\gra=1/2$, we have
\begin{equation}
 \|P\grph\|_{\grr_{1},1/2} \; \leq \; c_{1}\grr^{-3/2}\grs^{-2n}
 \|\grph\|_{\grr,0} .
\end{equation}
\end{lemma}
In what follows we shall increase $c_{a}$ and $s(a)$ a finite number
of times, but keep the same notation $K(a)$.

%8888888888888888888888888888888888888888888888888888888888888888888888888888
\section{Estimates for $F$.}
\setcounter{equation}{0}
In sections 5 and 6 we have derived estimates for $h_{1}$ and
$\barr{\del}_{X^{1}}Z_{1}$ in terms of $h$, $\barr{\del}_{X}Z$, and
$F$.  Now we use the results of section 7 to estimate $F$ and to
develop these results further.

Our basic iteration step will be the following.  Given $Z$ and
$\barr{\del}_{X}Z$ on $D_{\grr}(h)$ as in lemma 2.1, we form
$P\barr{\del}_{X}Z$ on $D_{\grr}(h)$ and estimate it on $D_{\grr(1-\grs)}(h)$,
$0<\grs<1/2$, using lemma 7.2.  For $0<t\leq\hat{c}^{-1}\grr\grs$,
we define and estimate $F=-S_{t}P\barr{\del}_{X}Z$ and the projected
map $f=I+f_{(2)}$ on $D_{\grr(1-\grs)^{2}}(h)$ using  (7.2).  Then
lemma 4.2 gives $g=I+g_{(2)}$, $Z_{1}$, $h_{1}$ on
$D_{\grr(1-\grs)^{3}}(h)$, which can be estimated on
$D_{\grr(1-\grs)^{4}}(h)$.  This last domain contains
$D_{\grr(1-\grs)^{3}(1-2\grs)}(h_{1})$, which in turn contains
$D_{\grr_{1}}(h_{1})$, $\grr_{1}=\grr(1-5\grs)$, $0<\grs<1/5$, since
$(1-\grs)^{3}(1-2\grs)\geq(1-3\grs)(1-2\grs)\geq(1-5\grs)$.  We may
now take norms $\|\cdot\|_{\grr_{1},a}$ over $D_{\grr_{1}}(h_{1})$ on
the left hand sides of the estimates.

For convenience we put
\begin{equation}
     \tilde{\grr}(j)=\grr(1-\grs)^{j}, 0<j\leq 4,
\end{equation}
relative to norms on the domains defined by the function $h$.
From (7.2) and (7.10) we get
\begin{equation}
  \|F\|_{\tilde{\grr}(2),1} \; \leq \; c_{1}\grr^{-5/2}\grs^{-2n}
   t^{-1/2}\|\barr{\del}_{X}Z\|_{\grr,0}.
\end{equation}
Thus the condition (4.17) in lemma 4.2 will hold if
\begin{equation}
  t^{-1/2}\|\barr{\del}_{X}Z\|_{\grr,1} \; \leq \; \grg_{1}\grr^{7/2}\grs^{2n+1},
\end{equation}
where we have replaced $\grg_{1}$ by a possibly smaller positive constant.
This will also guarantee the condition in lemma 6.2, as $t\leq1$.

Combining (7.2) and (7.8) gives, with $b\leq a$ and $b\leq m$, and a
possibly larger $K(a)$ of the same form (7.8),
\begin{eqnarray}
  \|F\|_{\tilde{\grr}(2),a} & \leq & c_{a}\grr^{-a}t^{b-a}
   \|P\barr{\del}_{X}Z\|_{\grr,b} \\    & \leq &  K(a)t^{b-a}
  \{ \|\barr{\del}_{X}Z\|_{\grr,b}+ \|h\|_{\grr,b+2}
   \|\barr{\del}_{X}Z\|_{\grr,0} \},  \nonumber  \\
  \|F\|_{\tilde{\grr}(2),1} & \leq & K(1)(1+\|h\|_{\grr,3})
   \|\barr{\del}_{X}Z\|_{\grr,1}.  \nonumber
\end{eqnarray}
We may instead use (7.2) and (7.9), with $b=l+(1/2)\leq a$,
$l\in\mathbf{Z}$, $l\leq m$.    This gives the following
alternative  estimate.
\begin{equation}
  \|F\|_{\tilde{\grr}(2),a} \; \leq \;  K(a)t^{l+(1/2)-a}
\{ (1+\|h\|_{\grr,\grb})\|\barr{\del}_{X}Z\|_{\grr,l}+
   \|h\|_{\grr,l+\grb}\|\barr{\del}_{X}Z\|_{\grr,0} \}.
\end{equation}

For the map $f$ we have, using (5.8), (8.2), (8.4),
\begin{eqnarray}
  \|f_{(2)}\|_{\tilde{\grr}(2),a} & \leq & K(a) \{
   t^{-1/2}\|\barr{\del}_{X}Z\|_{\grr,0} \|h\|_{\grr,a} +  \\ &  & + t^{b-a}
   [\|\barr{\del}_{X}Z\|_{\grr,b}+ \|h\|_{\grr,b+2}
   \|\barr{\del}_{X}Z\|_{\grr,0}]  \}, \nonumber
\end{eqnarray}
where $b\leq a$, and we recopy (5.7) as
\begin{equation}
   \|g_{(2)}\|_{\grr_{1},a} \; \leq \;
   \|g_{(2)}\|_{\tilde{\grr}(3),a} \; \leq \; c_{a}\grr^{-4(a+2)}
   \|f_{(2)}\|_{\tilde{\grr}(2),a}.
\end{equation}

To refine the estimates for $h_{1}$, we now take $\grr_{1}=\grr(1-5\grs)$
on the left-hand sides in (5.9) and (5.10), and
norms $\|\cdot\|_{\tilde{\grr}(2),*}$ on the
right-hand sides.  Using (8.2) and either (8.4) or (8.5), we get the
following coarse and fine estimates for $h_{1}$.
\begin{lemma}
Relative to the above described domains, and $1\leq a < \infty$,
$b\leq a$, $b\leq m$, or $k\in\mathbf{Z}$, $k+(1/2)\leq a$, $k\leq m$,
we have
\begin{eqnarray}
 \|h_{1}\|_{\grr_{1},a} & \leq & K(a) \{\|h\|_{\grr,a} + \\
   &  & + t^{b-a}(\|\barr{\del}_{X}Z\|_{\grr,b} +
        \|h\|_{\grr,b+2}\|\barr{\del}_{X}Z\|_{\grr,0}) \}, \nonumber \\
 \|h_{1}\|_{\grr_{1},a} & \leq & K(a) \{\|h\|_{\grr,a} + \\
   &  & + t^{k+(1/2)-a}[(1+\|h\|_{\grr,\grb})
      \|\barr{\del}_{X}Z\|_{\grr,k} +
        \|h\|_{\grr,k+\grb}\|\barr{\del}_{X}Z\|_{\grr,0} ] \}, \nonumber
\end{eqnarray}
\begin{eqnarray}
\|h_{1}-h\|_{\grr_{1},a} & \leq & K(a)  \{
     t^{-1/2}\|\barr{\del}_{X}Z\|_{\grr,0} \|h\|_{\grr,a+1} + \\
   &  & + t^{b-a}(\|\barr{\del}_{X}Z\|_{\grr,b} +
        \|h\|_{\grr,b+2}\|\barr{\del}_{X}Z\|_{\grr,0}) \}. \nonumber
\end{eqnarray}
\end{lemma}

For a coarse estimate of the new error $\barr{\del}_{X^{1}}Z_{1}$,
we may simplify (6.3) to
\begin{equation}
\begin{array}{rcl}
  \|\barr{\del}_{X^{1}}Z_{1}\|_{\grr_{1},a} & \leq & c_{a}\grr^{-a_{2}} \{
  \|\barr{\del}_{X}(Z+F)\|_{\tilde{\grr}(2),a} +
       \|F\|_{\tilde{\grr}(2),1}\|\barr{\del}_{X}Z\|_{\grr,a} + \\
  &  & +(\|\barr{\del}_{X}Z\|_{\grr,1} +\|F\|_{\tilde{\grr}(2),1})
   (\|F\|_{\tilde{\grr}(2),a+1} +
      \|F\|_{\tilde{\grr}(2),1}\|h\|_{\grr,a+1}) \}.
\end{array}
\end{equation}
%=================================================
%We use lemma 6.1 to estimate the first term,
%\begin{eqnarray}
%      \|\barr{\del}_{X}(Z+F)\|_{\grr,a} & \leq &
%  \|\barr{\del}_{X}Z\|_{\grr,a} + \|\barr{\del}_{X}F\|_{\grr,a} \\
% & \leq & \|\barr{\del}_{X}Z\|_{\grr,a} + c_{a}\grr^{-a_{2}}
%\{ \|F\|_{\grr,a+1} + \nonumber \\ &  &
%    +\|F\|_{\grr,1}(\|\barr{\del}_{X}Z\|_{\grr,a} +
%   \|h\|_{\grr,a+1}) \}. \nonumber
%\end{eqnarray}
Using (8.4) with $b=a<a+1$ to estimate $\|F\|_{\tilde{\grr}(2),a+1} $
and with $b=a=1$ to estimate $\|F\|_{\tilde{\grr}(2),1} $, and
combining some constants gives the following estimate.
\begin{lemma}
Assuming the above and $1\leq a\leq m$, we have
 \begin{eqnarray}
 \|\barr{\del}_{X^{1}}Z_{1}\|_{\grr_{1},a} & \leq & K(a+1)
  \{\|\barr{\del}_{X}(Z+F)\|_{\tilde{\grr}(2),a} + (1+\|h\|_{\grr,3})^{2}\cdot
   \\  &  & \nonumber  \cdot t^{-1}\|\barr{\del}_{X}Z\|_{\grr,1}
        (\|\barr{\del}_{X}Z\|_{\grr,a}+ \|\barr{\del}_{X}Z\|_{\grr,1}
         \|h\|_{\grr,a+2} ) \}.
 \end{eqnarray}
\end{lemma}

%9999999999999999999999999999999999999999999999999999999999999999999999999999

\section{Four-term estimates for $\barr{\del}_{X}(Z+F)$.}
\setcounter{equation}{0}
This is the central part  of the estimates, and is based on the 4 term
decomposition of section 3. We follow the scheme for the domains set
down at the beginning of the last section. We consider both $a$-norms
and $k$-norms, $a\in\mathbf{R}$, $k\in\mathbf{Z}$, $1\leq a,k\leq m$.
Also we use the notations (8.1) and (7.8) with
$K(a)=c_{a}(\grr\grs)^{-s(a)}$, where the constant $c_{a}$ and
polynomial $s(a)$ may be increased a finite number of times.

From (3.6) and (7.2) we have
\begin{equation}
  \|I_{1}\|_{\tilde{\grr}(1),a}=\|(I-S_{t})\barr{\del}_{X}Z\|_{\tilde{\grr}(1),a}
  \; \leq \; c_{m}\grr^{-a}t^{b-a}\|\barr{\del}_{X}Z\|_{\grr,b},
\end{equation}
for $0< t \leq \hat{c}^{-1}\grr\grs$, $a\leq b$.  We set $b=a+\grm \leq m$, then
\begin{equation}
  \|I_{1}\|_{\tilde{\grr}(1),a} \; \leq \; c_{m}\grr^{-a}
   t^{\grm}\|\barr{\del}_{X}Z\|_{\grr,a+\grm},
\end{equation}
which, of course, is also valid for $a=k$.

For the commutator term $I_{3}= [S_{t},\barr{\del}_{M}]P\barr{\del}_{X}Z$,
lemma 7.1 gives
\begin{eqnarray}
 \|I_{3}\|_{\tilde{\grr}(2),k} & \leq & c_{m}\grr^{-k-\gra}t^{\gra}
 \{\|P\barr{\del}_{X}Z\|_{\tilde{\grr}(1),k+\gra} + \|h\|_{\tilde{\grr}(1),k+2}
  \|P\barr{\del}_{X}Z\|_{\tilde{\grr}(1),\gra}\} , \\
 \|I_{3}\|_{\tilde{\grr}(2),a} & \leq & c_{m}\grr^{-a}
 \{\|P\barr{\del}_{X}Z\|_{\tilde{\grr}(1),a} + \|h\|_{\tilde{\grr}(1),a+2}
  \|P\barr{\del}_{X}Z\|_{\tilde{\grr}(1),0}\} .
\end{eqnarray}
In (9.3) we take $\gra =1/2$, $\grb=5/2$ and use (7.9), (7.10),
\begin{equation}
\begin{array}{rcl}
  \|P\barr{\del}_{X}Z\|_{\tilde{\grr}(1),k+1/2} & \leq & K(k)
  (1+\|h\|_{\grr,\grb})(\|\barr{\del}_{X}Z\|_{\grr,k} + \\  &  & +
  \|h\|_{\grr,k+\grb}\|\barr{\del}_{X}Z\|_{\grr,0}), \\
  \|P\barr{\del}_{X}Z\|_{\tilde{\grr}(1),1/2} & \leq & c_{1}
      \grr^{-3/2}\grs^{-2n}\|\barr{\del}_{X}Z\|_{\grr,0}.
\end{array}
\end{equation}
Thus,
\begin{equation}
   \|I_{3}\|_{\tilde{\grr}(2),k} \; \leq \; K(m)
     (1+\|h\|_{\grr,\grb})
       t^{1/2} \{ \|\barr{\del}_{X}Z\|_{\grr,k}
  + \|h\|_{\grr,k+\grb}\|\barr{\del}_{X}Z\|_{\grr,0} \}.
\end{equation}
In (9.4) we use instead (7.8) to get
\begin{equation}
  \|I_{3}\|_{\tilde{\grr}(2),a} \; \leq \; K(a)
  \{ \|\barr{\del}_{X}Z\|_{\grr,a}
  + \|h\|_{\grr,a+2}\|\barr{\del}_{X}Z\|_{\grr,0} \} .
\end{equation}

Next, $I_{2}\approx (X_{\barr{\gra}}-Y_{\barr{\gra}})S_{t}P\barr{\del}_{X}Z$,
so lemma 6.1 gives
\begin{equation}
\begin{array}{rcl}
  \|I_{2}\|_{\tilde{\grr}(2),a} &  \leq &  c_{a}\grr^{-2a}
   \{\|\barr{\del}_{X}Z\|_{\tilde{\grr}(2),0}
  \|S_{t}P\barr{\del}_{X}Z\|_{\tilde{\grr}(2),a+1} + \\  &  &
   (\|\barr{\del}_{X}Z\|_{\tilde{\grr}(2),0}\|h\|_{\tilde{\grr}(2),a+1} +
   \|\barr{\del}_{X}Z\|_{\tilde{\grr}(2),a})
       \|S_{t}P\barr{\del}_{X}Z\|_{\tilde{\grr}(2),1}\}.
\end{array}
\end{equation}
First we take $a=k$ in (9.8), and use (7.2) with $a=k+1$, $b=k+1/2$,
and (7.9),
\begin{equation}
\begin{array}{rcl}
  \|S_{t}P\barr{\del}_{X}Z\|_{\tilde{\grr}(2),k+1} &  \leq & c_{k}\grr^{-k-1}
   t^{-1/2}\|P\barr{\del}_{X}Z\|_{\tilde{\grr}(1),k+1/2}  \\ &  \leq &
          K(k)(1+\|h\|_{\grr,\grb})t^{-1/2}
         ( \|\barr{\del}_{X}Z\|_{\grr,k} + \\ &  & +
  \|h\|_{\grr,k+\grb}\|\barr{\del}_{X}Z\|_{\grr,0} ), \\
  \|S_{t}P\barr{\del}_{X}Z\|_{\tilde{\grr}(2),1} &  \leq & K(1)
  (1+\|h\|_{\grr,\grb})^{2}t^{-1/2}\|\barr{\del}_{X}Z\|_{\grr,0}.
\end{array}
\end{equation}
From this we get
\begin{equation}
\begin{array}{rcl}
  \|I_{2}\|_{\tilde{\grr}(2),k} &  \leq & K(k)
            (1+\|h\|_{\grr,\grb})^{2}
         t^{-1/2}\|\barr{\del}_{X}Z\|_{\grr,0}
      \{\|\barr{\del}_{X}Z\|_{\grr,k} + \\ &  & +
   \|h\|_{\grr,k+\grb}\|\barr{\del}_{X}Z\|_{\grr,0}\}.
\end{array}
\end{equation}
With $a$ real in (9.8) we use $b=a<a+1$ in (7.2) and (7.8),
\begin{eqnarray}
  \|S_{t}P\barr{\del}_{X}Z\|_{\tilde{\grr}(2),a+1} &  \leq &  K(a)t^{-1}
  \{ \|\barr{\del}_{X}Z\|_{\grr,a} +\|h\|_{\grr,a+2}
    \|\barr{\del}_{X}Z\|_{\grr,0} \} , \\
  \|S_{t}P\barr{\del}_{X}Z\|_{\tilde{\grr}(2),1} &  \leq & K(0)t^{-1}
   \|\barr{\del}_{X}Z\|_{\grr,0} .
\end{eqnarray}
This results in
\begin{equation}
  \|I_{2}\|_{\tilde{\grr}(2),a} \;   \leq \;  K(a) t^{-1}
  \|\barr{\del}_{X}Z\|_{\grr,0} \{\|\barr{\del}_{X}Z\|_{\grr,a} +
   \|h\|_{\grr,a+2} \|\barr{\del}_{X}Z\|_{\grr,0} \} .
\end{equation}

Finally we estimate $I_{4}=S_{t}Q\barr{\del}_{M}\barr{\del}_{X}Z$.
With $a=k\geq1$, $b = k-1/2$ in (7.2), and (7.9), we get
\begin{equation}
\begin{array}{rcl}
  \|I_{4}\|_{\tilde{\grr}(2),k} & \leq & c_{k}\grr^{-k}t^{b-k}
     \|Q\barr{\del}_{M}\barr{\del}_{X}Z\|_{\tilde{\grr}(1),b} \\
  & \leq & K(k)  (1 + \|h\|_{\grr,\grb})
t^{-1/2}\{\|\barr{\del}_{M}\barr{\del}_{X}Z\|_{\grr,k-1} + \\ &  & +
  \|h\|_{\grr,k-1+\grb}\|\barr{\del}_{M}\barr{\del}_{X}Z\|_{\grr,0}\}.
\end{array}
\end{equation}
With $b=a-1$ in (7.2), and (7.8), we get
\begin{equation}
  \|I_{4}\|_{\tilde{\grr}(2),a} \; \leq \; K(a)t^{-1} \{
  \|\barr{\del}_{M}\barr{\del}_{X}Z\|_{\grr,a-1} +
  \|h\|_{\grr,a+1}\|\barr{\del}_{M}\barr{\del}_{X}Z\|_{\grr,0} \}.
\end{equation}
Using lemma 6.1, we get
\begin{equation}
\begin{array}{rcl}
  \|\barr{\del}_{M}\barr{\del}_{X}Z\|_{\grr,a-1} & = &
  \|(\barr{\del}_{M}-\barr{\del}_{X})\barr{\del}_{X}Z\|_{\grr,a-1} \\
  & \leq & c_{a}\grr^{-2(a-1)} \{ \|\barr{\del}_{X}Z\|_{\grr,0}
      \|\barr{\del}_{X}Z\|_{\grr,a} + \\ &  & +
  (\|\barr{\del}_{X}Z\|_{\grr,0}\|h\|_{\grr,a} + \|\barr{\del}_{X}Z\|_{\grr,a-1})
   \|\barr{\del}_{X}Z\|_{\grr,1}  \} .
\end{array}
\end{equation}
We may use the convexity estimate (5.3) to absorb the last term into the
first.  We get
\begin{equation}
\begin{array}{rcl}
   \|\barr{\del}_{M}\barr{\del}_{X}Z\|_{\grr,a-1} & \leq &
   c_{a}\grr^{-3a} \|\barr{\del}_{X}Z\|_{\grr,0}
   (\|\barr{\del}_{X}Z\|_{\grr,a} + \|h\|_{\grr,a}
   \|\barr{\del}_{X}Z\|_{\grr,1} ) , \\
  \|\barr{\del}_{M}\barr{\del}_{X}Z\|_{\grr,0} & \leq & c_{0}
   \|\barr{\del}_{X}Z\|_{\grr,0}\|\barr{\del}_{X}Z\|_{\grr,1} .
\end{array}
\end{equation}
This leads to
\begin{eqnarray}
  \|I_{4}\|_{\tilde{\grr}(2),k} & \leq & K(k)
    (1+\|h\|_{\grr,\grb})t^{-1/2}\|\barr{\del}_{X}Z\|_{\grr,0}
       \{\|\barr{\del}_{X}Z\|_{\grr,k} +    \\  &  & + \nonumber
     \|h\|_{\grr,k-1+\grb}\|\barr{\del}_{X}Z\|_{\grr,1} \} ; \\
  \|I_{4}\|_{\tilde{\grr}(2),a} & \leq & K(a) t^{-1}
    \|\barr{\del}_{X}Z\|_{\grr,0} \{ \|\barr{\del}_{X}Z\|_{\grr,a} +
     \|h\|_{\grr,a+1}\|\barr{\del}_{X}Z\|_{\grr,1} \} .
\end{eqnarray}

Combining (9.2) with $a=k$, (9.6), (9.10), and (9.18),
and simplifying slightly gives the following fine four-term estimate.
\begin{lemma}
With $k\in\mathbf{Z}$, $1\leq k$,  $k+\grm\leq m$, $\grb=5/2$,
we have
\begin{equation}
\begin{array}{rcl}
  \|\barr{\del}_{X}(Z+F)\|_{\tilde{\grr}(2),k} & \leq & K(m)
        (1 + \|h\|_{\grr,\grb})^{2}  \{
    t^{\grm}\|\barr{\del}_{X}Z\|_{\grr,k+\grm} + \\ &  &  +
     t^{1/2}(\|\barr{\del}_{X}Z\|_{\grr,k} +
       \|h\|_{\grr,k+\grb}\|\barr{\del}_{X}Z\|_{\grr,0}) + \\  & & +
     t^{-1/2}\|\barr{\del}_{X}Z\|_{\grr,0}
      (\|\barr{\del}_{X}Z\|_{\grr,k} +
       \|h\|_{\grr,k+\grb}\|\barr{\del}_{X}Z\|_{\grr,1}) \} ;
\end{array}
\end{equation}
\end{lemma}

Combining (9.2) with $\grm=0$, (9.7), (9.13), and (9.19) gives
the following coarse four-term estimate.
\begin{lemma} For $1\leq a\leq m$,
\begin{eqnarray}
  \|\barr{\del}_{X}(Z+F)\|_{\tilde{\grr}(2),a} & \leq & K(m)
  (1+ t^{-1}\|\barr{\del}_{X}Z\|_{\grr,0})\cdot \\ &  &  \nonumber
   \cdot \{ \|\barr{\del}_{X}Z\|_{\grr,a} + \|h\|_{\grr,a+2}
     \|\barr{\del}_{X}Z\|_{\grr,1} \}.
\end{eqnarray}
\end{lemma}

In the general estimate of lemma 6.2,  we take $\grr_{1}=\grr(1-5\grs)$ and the
domain $D_{\grr_{1}}(h_{1})$ on the left hand side.  The first term on
the right is estimated by (9.20) for the following, which constitutes
the fine estimate for $\barr{\del}_{X^{1}}Z_{1}$.
\begin{lemma}
\begin{equation}
\begin{array}{rcl}
  \|\barr{\del}_{X^{1}}Z_{1}\|_{\grr_{1},k} & \leq & K(m)
           (1+\|h\|_{\grr,\grb})^{2} \{
  t^{\grm}\|\barr{\del}_{X}Z\|_{\grr,k+\grm} + \\  &  & +
        (t^{1/2} + t^{-1/2}\|\barr{\del}_{X}Z\|_{\grr,1} +
         t^{-1}\|\barr{\del}_{X}Z\|_{\grr,0} )\cdot \\  &  & \cdot
     (\|\barr{\del}_{X}Z\|_{\grr,k} +
       \|h\|_{\grr,k+\grb}\|\barr{\del}_{X}Z\|_{\grr,0}) \} .
\end{array}
\end{equation}
\end{lemma}
\begin{proof}
We must estimate the remaining terms in (6.3) with $a=k$.  We use
(8.4) with $a=b=k$ for $\|F\|_{\tilde{\grr}(2),k}$, and (8.3) for
 $\|F\|_{\tilde{\grr}(2),1}$.
This gives
\begin{eqnarray}
  \|F\|_{\tilde{\grr}(2), k} + \|F\|_{\tilde{\grr}(2),1}\|h\|_{\grr,k}
    & \leq & K(k) \{
   \|\barr{\del}_{X}Z\|_{\grr,k} + \\ &  & \nonumber
  +2t^{-1/2}\|\barr{\del}_{X}Z\|_{\grr,0}\|h\|_{\grr,k+2} \} , \\
 \|\barr{\del}_{X}Z\|_{\grr,0} + \|F\|_{\tilde{\grr}(2),1} & \leq &
  2K(1)t^{-1/2}\|\barr{\del}_{X}Z\|_{\grr,0} .
\end{eqnarray}
For $\|F\|_{\tilde{\grr}(2),k+1}$ we use (8.5) with $a=k+1$, $l=k$,
to get
\begin{eqnarray}
  \|F\|_{\tilde{\grr}(2) ,k+1} + \|F\|_{\tilde{\grr}(2),1}\|h\|_{\grr,k+1}
   & \leq & K(k+1)(1+\|h\|_{\grr,\grb})t^{-1/2}\cdot \\ &  & \nonumber
   \cdot(\|\barr{\del}_{X}Z\|_{\grr,k} +
     +  \|h\|_{\grr,k+\grb}\|\barr{\del}_{X}Z\|_{\grr,0}) \} .
\end{eqnarray}
Combining gives the lemma.
\end{proof}

For a coarse estimate of the error, we combine lemmas 8.2 and  9.2,
and simplify.
\begin{lemma}
For $1\leq a\leq m$,
\begin{eqnarray}
    \|\barr{\del}_{X^{1}}Z_{1}\|_{\grr_{1},a} & \leq & K(m+1)
    (1 + (1+\|h\|_{\grr,3})^{2}t^{-1}\|\barr{\del}_{X}Z\|_{\grr,1} )\cdot
  \\ &  & \nonumber \cdot
    (\|\barr{\del}_{X}Z\|_{\grr,a}+
                  \|h\|_{\grr,a+2}\|\barr{\del}_{X}Z\|_{\grr,1}).
\end{eqnarray}
\end{lemma}

%10-10-10-10-10-10-10-10-10-10-10-10-10-10-10-10-10-10-10-10-10-10-10-10-10

\section{The sequence of embeddings. Summary.}
\setcounter{equation}{0}
Now we begin the process of inductively constructing a sequence of
approximately holomorphic embeddings and showing convergence.  For $j\geq0$
we set, according to the second paragraph of section 8,
\begin{equation}
  \grr_{j+1}=\grr_{j}(1-5\grs_{j}), \; \;
   \tilde{\grr}_{j}(l)=\grr_{j}(1-\grs_{j})^{l}, \; \;
   \grs_{j+1}=5^{-1}\grs_{j},
\end{equation}
with $\grs_{0}=5^{-2}$, $0<l\leq4$, and $\grr_{0}>0$  free to be determined. We
choose $\grr_{0}\leq1$ so that $D_{\grr_{0}}\subset D^{0}_{1}$, where
$D^{0}_{1}$ is the domain of lemma 1.2 . Clearly the $\grr_{j}$ decrease
to a positive limit $\grr_{\infty}>0$. We also set
\begin{equation}
  K_{j}(a) = c_{a}(\grr_{j}\grs_{j})^{-s(a)}; \; \;
  N_{j}(a)=1+\|h_{j}\|_{\grr_{j},a}; \; \;
  \grd_{j}(a) = \|\barr{\del}_{X^{j}}Z_{j}\|_{\grr_{j},a},
\end{equation}
where we must take $a\leq m$ in the last.
We note that $h_{0}=0$, and $\grd_{0}(m)$ can be made arbitrarily
small by non-isotropic dilation.  We readily see that
\begin{equation}
  K_{j+1}(a) \; \leq \; \hat{c}_{a}K_{j}(a).
\end{equation}

We shall choose the smoothing parameters $t_{j}$ by
\begin{equation}
  t_{j+1} = t^{\grk}_{j}, \; \; 1<\grk, \; \; 0<t_{0}<1,
\end{equation}
with $t_{0}$ sufficiently small. Then the $t_{j}$
strictly decrease rapidly to zero.   Notice that
$\hat{c}t_{j+1}/\grr_{j+1}\grs_{j+1} \, \leq \,
  (\hat{c}t_{j}/\grr_{j}\grs_{j})t^{\grk-1}_{0}(5/1-5\grs_{0})$,
so that
\begin{equation}
  0 < t_{j} < \hat{c}^{-1}\grr_{j}\grs_{j}, \; (\hat{c} = 5/\sqrt{2}),
\end{equation}
for all $j$, if $t_{0}$ is sufficiently small, $\grr_{0}=1$ being fixed.
(We shall have to shrink $t_{0}$ a finite number of times.)

We want to  construct $h_{j}$,  $Z_{j}\in C^{\infty}(D_{\grr_{j}})$, and
$X^{j}_{\gra}\in C^{m}(D_{\grr_{j}})$, as in (1.4), (1.5) and (2.2).
To pass from $j$ to $j+1$, we need to make the
inductive assumption
\begin{equation}
    \|h_{j}\|_{\grr_{j},2} \, \leq  \, \grg_{0} ,
\end{equation}
in order to apply lemma 2.1 and to make use of the results of section 7.
Then we may use the solution operator $P=P_{j}$ and smoothing
operator $S_{t_{j}}$ to construct
$F_{j}=-S_{t_{j}}P_{j}\barr{\del}_{X^{j}}Z_{j}$.

To construct our sequences, we shall choose a suitable integer $k$,
$1\leq k\leq m$, and a real number $a$, $2 \leq a $,
$k\leq a$.  The goal is to get convergence of $h_{j}$, $Z_{j}$ in
the H\"older class $C^{a}$, while the error $\grd_{j}(k)$ goes to
zero rapidly.  We make the second inductive assumption,
\begin{equation}
  t^{-s}_{j}\grd_{j}(k) \, \leq \, 1,
\end{equation}
for a suitable choice of $s> 1/2$.  Given this, we claim that the
conditions
\begin{equation}
   t_{j}^{-1/2}\grd_{j}(0)\leq\grg_{1}\grr_{j}^{7/2}\grs_{j}^{2n+1},  \; \; \;
   \grd_{j}(1)\leq 1,
\end{equation}
will hold for all $j$, if $t_{0}>0$ is  sufficiently small.  The
second is clear.  To see the first,  put
$Q_{j}=t_{j}^{s-1/2}\grr_{j}^{-7/2}\grs_{j}^{-2n-1}$.  By (10.7) we need
$Q_{j}\leq\grg_{1}$ for all $j$. Since $\grk > 1$, $s>1/2$ and the
$t_{j}$, $\grs_{j}$ decrease, we easily see that $Q_{0}\leq\grg_{1}$
and $Q_{j+1}/Q_{j}\leq1$, if $t_{0}$ is sufficiently small.

We start the process at $j=0$, using lemma 1.2, so that $h_{0}=0$.
Once $s$, $\grk$  and $t_{0}>0$ have been fixed,
we may apply non-isotropic dilation to make $\grd_{0}(m)$ as small
as is needed.  This will give (10.6), (10.7) for $j=0$.  We then construct
our sequences $h_{j}$, $Z_{j}$ of approximate CR embeddings, and
verify (10.6), (10.7), inductively.

From (8.6), (8.7), lemma 8.1, lemma 9.3, and lemma 9.4, respectively,
we have the following. They summarize the main estimates that have
been derived up to this point.  In them $\grb=5/2$,  $b\leq a$,
$b\leq m\in\mathbf{R}$,
$k,l$ are integers $1\leq k,l\leq m$, $k+\grm\leq m$,
$l+(1/2)\leq a$ .
\begin{eqnarray}
\|f_{j(2)}\|_{\tilde{\grr_{j}}(2),a} & \leq &
     K_{j}(a)\{t_{j}^{-1/2}\grd_{j}(0)N_{j}(a) + \\ &  & \nonumber
      + t_{j}^{b-a}(\grd_{j}(b)+ N_{j}(b+2)\grd_{j}(0)) \}, \\
\|g_{j(2)}\|_{\grr_{j+1},a} & \leq &  c_{a}\grr_{j}^{-4(a+2)}
                         \|f_{j(2)}\|_{\tilde{\grr_{j}}(2),a},
\end{eqnarray}
\begin{equation}
\|h_{j+1}-h_{j}\|_{\grr_{j+1},a} \; \leq \; K_{j}(a)
     \{t_{j}^{-1/2}\grd_{j}(0)N_{j}(a+1) +
                  t_{j}^{b-a}(\grd_{j}(b)+N_{j}(b+2)\grd_{j}(0)) \},
\end{equation}
\begin{eqnarray}
  N_{j+1}(a)  & \leq & K_{j}(a)\{N_{j}(a) +
     t_{j}^{b-a}(\grd_{j}(b)+N_{j}(b+2)\grd_{j}(0)) \} , \\
  N_{j+1}(a)  & \leq & K_{j}(a)\{N_{j}(a) +
     t_{j}^{l+(1/2)-a}(N_{j}(\grb)\grd_{j}(l)+N_{j}(l+\grb)\grd_{j}(0)) \} ,
\end{eqnarray}
\begin{eqnarray}
\grd_{j+1}(k) & \leq & K_{j}(m)N_{j}(\grb)^{2} \{
       t_{j}^{\grm}\grd_{j}(k+\grm) + \\ & & +
       (t_{j}^{1/2} + t_{j}^{-1/2}\grd_{j}(1)+t_{j}^{-1}\grd_{j}(0))
        \grd_{j}(k)N_{j}(k+\grb) \} , \nonumber
\end{eqnarray}
\begin{eqnarray}
\grd_{j+1}(k+\grm) & \leq & K_{j}(m+1)(1+N_{j}(3)^{2}t_{j}^{-1}\grd_{j}(1))\cdot
\\ &  &  \nonumber \cdot \{ \grd_{j}(k+\grm) + N_{j}(k+\grm+2)\grd_{j}(1)\} , \\
\grd_{j+1}(m) & \leq & K_{j}(m+1)(1+N_{j}(3)^{2}t_{j}^{-1}\grd_{j}(1))\cdot
  \\ &  &  \nonumber \cdot \{ \grd_{j}(m) + N_{j}(m+2)\grd_{j}(1)\} .
\end{eqnarray}

%11-11-11-11-11-11-11-11-11-11-11-11-11-11-11-11-11-11-11-11-11-11-11-11-11-11

\section{The main inductive hypotheses.}
\setcounter{equation}{0}

In this section we verify inductively (10.6) and (10.7) for all $j$.
For this we need to control the possible growth of the following norms.
\begin{lemma}
  If $t_{j}^{-s}\grd_{j}(k)\leq1$, with $s\geq2$, then
\begin{equation}
  N_{j+1}(a) \; \leq \; 3K_{j}(a)N_{j}(a) ,
\end{equation}
for $a=3$, $a=\grb$, $a=k+2$,  $a=k+\grb$.
\end{lemma}
\begin{proof}
  For the case $a=3$, we use (10.12) with $a=3$ and $b=1$, and (10.7).
This gives
$ N_{j+1}(3) \; \leq \; K_{j}(3)N_{j}(3)\{1+2t_{j}^{s-2}\}$, hence the
estimate. For $a=k+2$, we use (10.12) with $b=k$. For $a=\grb$ or
$a=k+\grb$, we use (10.13) with $l=0$, or $l=k$.
\end{proof}

Now we assume that (10.6), (10.7) hold for all $j\leq l$ and
verify (10.6) for $j=l+1$.
\begin{lemma}
Assume that (10.6) and (10.7) with $s\geq2$ hold for all $j\leq l$.
Then $\|h_{l+1}\|_{\grr_{l+1},2} \leq \grg_{0}$, independently of $l$, if
$t_{0}>0$ is chosen sufficiently small.
\end{lemma}
\begin{proof}
For the proof we apply (10.11) with $a=2$, $b=1$, and use (10.7).
This gives
\begin{equation}
 \|h_{j+1}-h_{j}\|_{\grr_{j+1},2} \; \leq \;
 3K_{j}(2)N_{j}(3)t_{j}^{s-1} \equiv P_{j}.
\end{equation}
Since $h_{0}=0$, it suffices to show $\sum_{j=0}^{l}P_{j}\leq\grg_{0}$.
From (11.1) with $a=3$, it follows that
\begin{eqnarray}
  P_{j+1}/P_{j} & \leq & 3\hat{c}_{2}K_{j}(3)t_{j}^{(\grk-1)(s-1)}
    \equiv Q_{j}, \\
  Q_{j+1}/Q_{j} & \leq & \hat{c}_{3}t_{j}^{(\grk-1)^{2}(s-1)} .
\end{eqnarray}
Since $\grr_{0}$, $\grs_{0}$ are fixed, this we see that $Q_{0} < 1$,
and $Q_{j+1}\leq\gre Q_{j}$, say $\gre=1/2$, if $t_{0}>0$ is chosen
sufficiently small, depending on $\grk$, $s$. Then
$Q_{j}\leq\gre^{j}$, and for $j\geq1$,
$P_{j} \, \leq \, Q_{j-1}P_{j-1} \, \leq \cdots \leq
  \sqrt{\gre}^{j(j-1)}P_{0}$.
Thus
\begin{equation}
  \sum_{j=0}^{l} P_{j} \, \leq \, P_{0} +
    (\sum_{j=1}^{\infty}\sqrt{\gre}^{j(j-1)})P_{0} \,\leq \,
   P_{0}/(1-\sqrt{\gre}) .
\end{equation}
This will be less than or equal $\grg_{0}$, if $t_{0}$ is sufficiently
small, since $h_{0}=0$  and $N_{0}(3)=1$.
\end{proof}

It remains to achieve the main inductive assumption (10.7) for all $j$.
Thus we assume (10.7) holds for $j$ and verify it for $j+1$.
Since $s\geq2$, $t_{j}^{1/2}$ is largest of three terms, so  (10.14)
and (10.7) give
\begin{eqnarray}
  t^{-s}_{j+1}\grd_{j+1}(k) & \leq & a_{j} + B_{j}, \\
  a_{j}  & = & 3K_{j}(m)N_{j}(\grb)^{2}N_{j}(k+\grb)t^{(1-\grk)s +1/2}_{j},\\
  B_{j} & = & K_{j}(m)N_{j}(\grb)^{2}t^{\grm - \grk s}_{j}\grd_{j}(k+\grm) .
\end{eqnarray}
We want the two exponents of $t_{j}$ to be positive. Thus we shall choose
$\grm > \grk s$, $\gra=(1-\grk)s +1/2 > 0$, and $s=2$.  This
gives the restriction $1<\grk<5/4$. Thus we choose
\begin{equation}
    s=2 , \; \; \; \;  1 < \grk < 5/4 , \; \; \; \; \grm > \grk s .
\end{equation}

The first step is to make $a_{j}<1/2$. Using lemma 11.1
as in the proof of lemma 11.2, we get
\begin{equation}
\begin{array}{rcl}
  a_{j+1}/a_{j} & \leq & 27\hat{c}_{m}K_{j}(\grb)^{2}K_{j}(k+\grb)
     t_{j}^{(\grk-1)\gra} \equiv\tilde{a}_{j}, \\
  \tilde{a}_{j+1}/\tilde{a}_{j} & \leq &
     \hat{c}_{\grb}^{2}\hat{c}_{k+\grb} t_{j}^{(\grk-1)^{2}\gra}.
\end{array}
\end{equation}
Thus, if $t_{0}$ is sufficiently small, all $\tilde{a}_{j} < 1$,
and all $a_{j} < 1/2$. We shall use this kind of argument several
more times.

The second step is to make $B_{j} < 1/2$ for all $j$. By (10.14),
(10.7) and (10.16), and since $1+N_{j}(3)^{2}t_{j}^{s-1}\leq2$, if
$t_{0}$ is sufficiently small,  we have
\begin{eqnarray}
B_{j+1} & \leq & b_{j}B_{j} + E_{j}, \\
b_{j} & = & C_{j}t_{j}^{(\grk-1)(\grm-\grk s)}, \\
E_{j} & = & C_{j} N_{j}(\grb)^{2}
      N_{j}(k+\grm+2)t_{j}^{\grk(\grm-\grk s)+s},
\end{eqnarray}
where, for a common constant, we may take
$C_{j}=18\hat{c}_{m}K_{j}(\grb)^{2}K_{j}(m+1)^{2}$.
Since $\grm>\grk s$, both of the exponents are positive.
Since $C_{j+1}/C_{j}=\hat{C}_{0}$, we see that $b_{j}\leq 1/2$
for all $j$, if $t_{0}$ is sufficiently small.

If we can show that $E_{j} \leq 1/4$ for all $j$, then
$B_{j} < 1/2$ for all $j$. To check its possible growth, we use
(10.13) with $a=k+\grm+2$ and $l=k$.  With $\grg=\grk(\grm-\grk s)+s$,
we get
\begin{eqnarray}
  E_{j+1} & \leq & e_{j}E_{j} + f_{j} , \\
  e_{j} & = & C_{j}' t_{j}^{(\grk-1)\grg} , \\
  f_{j} & = & C_{j}''N_{j}(3)^{2}N_{j}(k+\grb)t_{j}^{\grk\grg+s-\grm-3/2},
\end{eqnarray}
where $C_{j}'=\hat{C}_{0}(3K_{j}(\grb))^{2}K_{j}(k+\grm+2)$ and
$C_{j}''= C_{j}'C_{j}$.  Clearly we can arrange $e_{j}< 1/2$ by
choosing $t_{0}$ sufficiently small. By lemma 11.1 we can make
$f_{j} < 1/8$ for all $j$, provided the exponent of $t_{j}$ is
positive. But a simple computation shows that
$\grk\grg+s-\grm-3/2 >0$, given (11.9).

Hence, if $t_{0}>0$ is sufficiently small, then $B_{j}<1/2$
for all $j$, if (11.19) holds. This proves the following.
\begin{lemma}
  Suppose that $s=2$, $1<\grk<5/4$ and $\grm>\grk s$.
Let $1\leq k\leq m$, $2\leq a$, $k\leq a$.
If $t_{0}>0$ and then $\grd_{0}(k)$ are taken sufficiently small,
(10.6) and (10.7) will hold for all $j$.
\end{lemma}
Explicitly, we require $\grm> \grk s$, $5/2 > \grk s > 2$, and
$k+\grm\leq m$.

%12-12-12-12-12-12-12-12-12-12-12-12-12-12-12-12-12-12-12-12-12-12-12-12-12-12

\section{Convergence in $C^{a}$.}
\setcounter{equation}{0}
Now we assume that $k\leq a$ and the other parameters are chosen as in
the last lemma.  We want to show that our sequence of approximate
CR-embeddings, $h_{j}$, $Z_{j}$ on $D_{\grr_{j}}(h_{j})$, converges in
$C^{a}$-norm on some neighborhood of 0. By lemma 2.1 all these domains
contain the ball of radius $\sqrt{2/3}\grr_{\infty}$. By (10.11) with
$b=k$, it suffices to show that
\begin{equation}
  \sum_{j=1}^{\infty}
     K_{j}(a)\grd_{j}(k)\{ t_{j}^{-1/2}N_{j}(a+1) + 2t_{j}^{k-a}N_{j}(k+2) \}
      < \infty .
\end{equation}
Using (10.7) it suffices to show both
\begin{equation}
  \sum_{j=1}^{\infty} K_{j}(a)N_{j}(a+1)t^{s-1/2}_{j} < \infty, \; \;
  \sum_{j=1}^{\infty} K_{j}(a)N_{j}(k+2)t^{k-a+s}_{j} < \infty .
\end{equation}
We apply the ratio test to the second, then first series.
\begin{eqnarray}
  \frac{K_{j+1}(a)N_{j+1}(k+2)t^{k-a+s}_{j+1}}
       {K_{j}(a)N_{j}(k+2)t^{k-a+s}_{j}}  & \leq &
  2\hat{c}_{j}(a)K_{j}(k+2)t_{j}^{(\grk-1)(s+k-a)} , \\
  \frac{K_{j+1}(a)N_{j+1}(a+1)t^{s-1/2}_{j+1}}
       {K_{j}(a)N_{j}(a+1)t^{s-1/2}_{j}}  & \leq &
   \hat{c}_{j}(a)K_{j}(a+1) t_{j}^{(\grk-1)(s-1/2)}\{ 1 + \\ & & \nonumber
   + t^{k-a-1+s}_{j}N_{j}(k+2)/N_{j}(a+1) \}.
\end{eqnarray}
With $a\geq k+1$, the right-hand sides will go to zero, if the
exponents of $t_{j}$ are positive.  Thus we need $k+s>a$ and
$k+s-1+(\grk-1)(s-1/2)>a$, that is $k+2>a$ and $k+1+(3/2)(\grk-1)>a$.
In particular, we can take $a=k+1$.

The restriction on the integer $m$ comes from $m\geq k+\grm$,
$k\geq 1$, $\grm > \grk s$, $5/2 > \grk s > 2$. Thus we need $m > 3$
to run the argument.  We put $m=3+\gry$, $\gry>0$.  Then we need
\begin{eqnarray}
   1 & < & \grk \; < \; \mbox{min}(5/4, (\gry +3-k)/2) \; \leq \;
   \mbox{min}(5/4,1+(\gry/2)) , \\
   2 & < & \grk s \; < \; \grm \; < \; m-k.
\end{eqnarray}
For $m > 3$, we fix such a $\grk$, and $\grm$,  and run the argument.

We summarize what has been achieved thus far.  We have constructed a
sequence of $C^{\infty}$-smooth real hypersurfaces
$M_{j}:y^{n}=|z'|^{2}+h_{j}(z',x^{n})$,
and embeddings $Z_{j}=(z',x^{n}+i(|z'|^{2}+h_{j}(z',x^{n}))$ of a
neighborhood $D_{\infty}$ of $0$ in the real hyperplane $Im(z^{n})=0$,
into $\mathbf{C}^{n}$.  Our original CR structure, or complex vector
frame field $X_{\gra}$ of class $C^{m}$
has also been transplanted to $D_{\infty}$, where it is subjected to
a sequence of diffeomorphisms and frame changes to get
the sequence $X_{\gra}^{j}$ of complex vector frame fields.
We have $X_{\gra}^{j} = Y_{\gra}^{j} +
 A_{\gra}^{\;\barr{\grb}} Y_{\barr{\grb}}^{j} + B_{\gra}
 \del_{x^{n}}$, as in (1.8), (1.9),
where the $Y_{\barr{\grb}}$, given by (1.7), are the tangential CR
operators to $M_{j}$.

We have shown that $h_{j}\rightarrow h_{\infty}$ and
$Z_{j}\rightarrow Z_{\infty}$ in $C^{a}(D_{\infty})$ norm, while
$X_{\barr{\gra}}^{j}Z_{j}\rightarrow0$, rapidly.  It follows that
$Y_{\gra}^{j}\rightarrow Y_{\gra}^{\infty}$ in $C^{a-1}=C^{k}$ norm.
By (1.9) $( A_{\gra}^{\;\barr{\grb}},B_{\gra})\rightarrow0$ (rapidly).
Hence, $X_{\gra}^{j}\rightarrow X_{\gra}^{\infty} \equiv
Y_{\gra}^{\infty}$, in $C^{k}$ norm.  These limiting vector fields are
of class $C^{a-1}=C^{k}$.

We \emph{claim} that the complex
vector frame field $X_{\gra}^{\infty} \equiv Y_{\gra}^{\infty}$
spans a CR structure equivalent to our original CR structure. The
equivalence will be given by the composition $\tilde{f}_{\infty}$ of
all the maps $f_{j}$. This will be shown to converge in the appropriate
norms in the final section.  Then $Z_{\infty}\circ\tilde{f}_{\infty} $
will be the required CR embedding.  This argument will yield the  following.
\begin{prop}
Let $X_{\barr{\gra}}$, $1\leq\gra\leq n-1$, be local complex
vector fields on $\mathbf{R}^{2n-1}$, $2n-1\geq7$, which represent
a formally integrable, strongly pseudoconvex CR structure of
class $C^{m}$, $m\in\mathbf{R}$, $3<m\leq\infty$.
Then there exists a local CR embedding.  It is  of class
$C^{a}$, for all $a$,  $0\leq a < m$.
\end{prop}
The arguments of this and the final section give the result with
$0\leq a < m-\grm+1$. The next section improves the regularity of the
solutions to that stated in the proposition and theorem 0.1.

%13-13-13-13-13-13-13-13-13-13-13-13-13-13-13-13-13-13-13-13-13-13-13-13-13-13

\section{Interpolation and higher regularity.}
\setcounter{equation}{0}

Now let $m_{0}\in\mathbf{R}$, $m_{0}>3$, $2\leq a=k+1\leq m_{0}-\grm+1$,
and $X_{\gra}$ of class $C^{m_{0}}$, and $h_{j},Z_{j}$ converging in
$C^{a}(D_{\infty})$ as in proposition 12.1.  Let $m\in\mathbf{Z}$ with
$m\geq m_{0}$.

We want to investigate the convergence of these same sequences also in
$C^{b}(D_{\infty})$, for
\begin{equation}
  a < b < m+ (1/2), \; \; b = \grl a + (1-\grl)(m+(1/2)),\; \;
  0 < \grl < 1 ,
\end{equation}
for any $m\geq m_{0}$ for which the original $X_{\gra}$ are of class
$C^{m}$. We shall apply the ratio test to the series
$\sum \|h_{j+1}-h_{j}\|_{b}$, where the norms  are over the domain
$D_{\infty}$.  For this we use the interpolation inequality (5.2) to get
\begin{eqnarray}
  \|h_{j+1}-h_{j}\|_{b} & \leq & c_{m}\grr_{j}^{-m}A^{\grl}_{j}B^{(1-\grl)}_{j}, \\
      A_{j} & = & \|h_{j+1}-h_{j}\|_{\grr_{j+1},a}, \\
      B_{j} & = & \|h_{j+1}-h_{j}\|_{\grr_{j+1},m+(1/2)} \\  \nonumber
           & \leq & N_{j+1}(m+(1/2)) + N_{j}(m+(1/2)).
\end{eqnarray}
We have $\grr_{j}\geq\grr_{\infty}>0$, and we know that $A_{j}$ goes
to zero rapidly.  In fact,as in (12.2), taking $a=k+1$ and $b=k$
in (10.11) and using (10.7) and lemma 11.1 gives (increasing $A_{j}$)
\begin{eqnarray}
   A_{j} & = & 2K_{j}(k+1)N_{j}(k+2)t_{j}^{s-1}, \\
   A_{j+1}/A_{j} & \leq & 3\hat{c}_{k+1}K_{j}(k+2)t_{j}^{(\grk-1)(s-1)}.
\end{eqnarray}
We must control the possible growth of $B_{j}$.

Taking $a=m+(1/2)$, then $a=m+\grb$, and $l=m$ in (10.13) and using
(10.7) gives
\begin{eqnarray}
  N_{j+1}(m+(1/2)) & \leq & K_{j}(m+(1/2)) \{N_{j}(m+(1/2)) + \\ &  &
     + N_{j}(\grb)\grd_{j}(m) + t_{j}^{s}N_{j}(m+\grb) \}. \nonumber  \\
  N_{j+1}(m+\grb) & \leq & K_{j}(m+\grb) \{N_{j}(m+\grb) + \\ &  &
   + t_{j}^{-2}[N_{j}(\grb)\grd_{j}(m)+ t_{j}^{s}N_{j}(m+\grb)] \}. \nonumber
\end{eqnarray}
By increasing $B_{j}$, we may write
\begin{eqnarray}
   B_{j} & = & E_{j} + F_{j} + G_{j} , \\
   E_{j} & = & 2 K_{j}(m+(1/2))N_{j}(m+(1/2)) , \\
   F_{j} & = & K_{j}(m+(1/2))N_{j}(\grb)\grd_{j}(m) , \\
   G_{j} & = & K_{j}(m+(1/2))N_{j}(m+\grb)t_{j}^{s} .
\end{eqnarray}
We next derive the following growth estimates.  The first follows directly
from (13.7) and the definitions.
\begin{eqnarray}
  E_{j+1} & \leq & 2\hat{c}_{m+(1/2)}K_{j}(m+(1/2)) B_{j} , \\
  F_{j+1} & \leq &
       12\hat{c}_{m+(1/2)}K_{j}(\grb)K_{j}(m+1)N_{j}(\grb)B_{j}. \\
  G_{j+1} & \leq & 2\hat{c}_{m+(1/2)}K_{j}(m+\grb)
     t_{j}^{2(\grk-1)} B_{j}.
\end{eqnarray}
For the third we use (13.8), $s=2$, and  $\grk s-2\geq0$.
For the second we use (10.16), (10.7), and
$1+N_{j}(3)^{2}t_{j}^{s-1}\leq2$, as arranged in section 11.  We get
\begin{equation}
  \grd_{j+1}(m) \; \leq \;2 K_{j}(m+1) \{ \grd_{j}(m) +
     N_{j}(m+2)t_{j}^{s} \} .
\end{equation}
From this we get
\begin{equation}
 F_{j+1} \; \leq \; 6\hat{c}_{m+(1/2)}K_{j}(\grb)K_{j}(m+1)
    (F_{j} + N_{j}(\grb)G_{j}),
\end{equation}
which gives (13.14).  Combining gives
\begin{equation}
  B_{j+1} \;  \leq \;
  16\hat{c}_{m+(1/2)}K_{j}(m+\grb)K_{j}(\grb)N_{j}(\grb)B_{j}.
\end{equation}

From (13.6) and (13.18) we see that
\begin{equation}
(A_{j+1}/A_{j})^{\grl}(B_{j+1}/B_{j})^{1-\grl}\, \leq \,
                         \tilde{C}_{j}t_{j}^{\gra\grl},
\end{equation}
where $\gra>0$,and $\tilde{C}_{j}$ is slowly growing. Since $0<\grl$, (13.19)
tends to zero, as $j\rightarrow \infty$.
Hence, we have convergence in $C^{b}(D_{\infty})$.

Thus, if our original vector fields $X_{\gra}$ are of class $C^{m}$,
$\infty>m\geq m_{0}>3$, then the limiting real hypersurface
$M_{\infty} : \; \;  y^{n} = |z'|^{2} + h_{\infty}(z',x^{n})$,
and CR embedding
 $ Z_{\infty} = (z',x^{n} + i(|z'|^{2} + h_{\infty}(z',x^{n}))),$
are of class $C^{b}$ for all $b < m+1/2$. In the $C^{\infty}$ case,
we may apply the same argument, with perhaps different constants
appearing in the estimates, for each $m>m_{0}$, to the fixed
sequence. Thus, $M_{\infty}$ and $Z_{\infty}$ are class $C^{\infty}$.
Combined with  proposition 12.1, this gives theorem 0.1.

In case $m\in\mathbf{R}$, $m > 3$, is not an integer, an entirely
similar but simpler argument, using the estimate for $N_{j}(m)$
gotten from (10.12) with $a=b=m$ gives that the embedding of
proposition 12.1 is of class $C^{a}$, $0\leq a < m$.

%14-14-14-14-14-14-14-14-14-14-14-14-14-14-14-14-14-14-14-14-14-14-14-14-14

\section{Composition of mappings.}
\setcounter{equation}{0}

In this section we show that the sequence of compositions of the
maps $f_{j}$ actually converges on some neighborhood of $0$ in
$\mathbf{R}^{2n-1}$.  The limiting map will provide a CR-equivalence of
our original structure $X_{\gra}$ and the embedded structure
$X_{\gra}^{\infty}$ of section 12. With it we shall get  the
solutions $z^{j}$ of (0.1) for the original vector fields $X_{\barr{\gra}}$.
The argument can be motivated by a similar but much simpler one given
in [22].

We consider the sequences of compositions of mappings,
\begin{equation}
\begin{array}{c}
  \tilde{f_{j}} = f_{j}\circ f_{j-1}\circ\cdots\circ f_{0} :
\tilde{U}_{j}\rightarrow \mathbf{R}^{2n-1} , \\
\tilde{g_{j}} = g_{0}\circ g_{1}\circ\cdots\circ g_{j} :
D_{\grr_{j+1}}\rightarrow D_{\grr_{0}},
\end{array}
\end{equation}
where $\tilde{U}_{j}=\tilde{g_{j}}(D_{\grr_{j+1}})$.  The domains
$\tilde{U}_{j}\subseteq U_{j}\subset D_{\tilde{\grr}_{j}(2)}$ (see lemma 4.1
and (8.1)) are decreasing, compact, smoothly bounded neighborhoods of $0$.
The domains $D_{\grr_{j}}$ are decreasing, compact, strictly convex,
and $D_{\grr_{j}}\supseteq B(\sqrt{2/3}\grr_{j})$, and
$D_{\infty}=\cap D_{\grr_{j}}\supseteq B(\sqrt{2/3}\grr_{\infty})$ is
a compact convex neighborhood of $0$.

To investigate the convergence of the sequences $\tilde{f}_{j}$,
$\tilde{g}_{j}$, we let $d$ denote Jacobian matrix, and $|\cdot|_{U}$
denote the sup of the matrix operator norm over the set $U$.  Then,
since $\tilde{g}_{j} = \tilde{g}_{j-1}\circ g_{j}$,

\begin{eqnarray}
  d\tilde{g}_{j} & = & (d\tilde{g}_{j-1}\circ g_{j})dg_{j} , \\
 |d\tilde{g}_{j}|_{\grr_{j+1}} & \leq & |d\tilde{g}_{j-1}|_{\grr_{j}}
    |dg_{j}|_{\grr_{j+1}} \leq \cdots \leq
      \prod _{i=0}^{j} ( 1 + |dg_{i(2)}|_{\grr_{i+1}} ).
\end{eqnarray}

The infinite product will converge iff
$\sum_{j=0}^{\infty}|dg_{j(2)}|_{\grr_{j+1}} \, < \, \infty$, which will
follow from
$\sum_{j=0}^{\infty}\|g_{j(2)}\|_{\grr_{j+1},k} \, < \, \infty$, since
$k\geq1$.
But this  follows from (10.10), and (10.9) with $a=b=k$, and (10.7), which give
\begin{equation}
   \|g_{j(2)}\|_{\grr_{j+1},k} \, \leq \, K_{j}(k)t_{j}^{s-1/2}N_{j}(k+2) ,
\end{equation}
and lemma 11.1, and the ratio test.  In particular, we have a uniform
bound, $|d\tilde{g}_{j}|_{\grr_{j+1}} \leq C_{1}$ for all $j$.

As in the proof of lemma 4.2,
\begin{eqnarray}
  \tilde{g}_{j}(x) - \tilde{g}_{j-1}(x) & = & \int_{t=0}^{1}
   d\tilde{g}_{j-1}(g_{jt}(x))[g_{j(2)}] , \\
  \|\tilde{g}_{j} - \tilde{g}_{j-1}\|_{\grr_{j+1},0} & \leq &
    C_{1}\|g_{j(2)}\|_{\grr_{j+1},0}.
\end{eqnarray}
It follows that $\tilde{g}_{j}$ will converge uniformly on $D_{\infty}$ to
a Lipschitz continuous mapping $\tilde{g}_{\infty}$.  We further estimate
$\tilde{g}_{j} = \tilde{g}_{j-1}\circ g_{j}$ on $D_{\infty}$,   using
the chain-rule estimate (5.5).  Combining some constants gives
\begin{equation}
   \|\tilde{g}_{j}\|_{k} \; \leq \; K_{j}(k)(\|\tilde{g}_{j-1}\|_{k} +
  C_{1}\|g_{j(2)}\|_{k}).
\end{equation}

We use this with $k=2$ to estimate the 1-norm of (14.5).  We use
the product-rule estimate (5.4) and the chain-rule estimate (5.5) for
$\tilde{g}_{j} = \tilde{g}_{j-1}\circ g_{j}$, as in the proof of lemma
(5.2).  On  $D_{\infty}$ we get
\begin{eqnarray}
  \|\tilde{g}_{j} - \tilde{g}_{j-1}\|_{1} & \leq & K_{j}(1)
   ( \|g_{j(2)}\|_{0}\max_{t}\|d\tilde{g}_{j-1}\circ g_{jt}\|_{1} +
    C_{1}\|g_{j(2)}\|_{1} ) , \\
  \|d\tilde{g}_{j-1}\circ g_{jt}\|_{1} & \leq &  K_{j}(1) (
   \|\tilde{g}_{j-1}\|_{2} + C_{1}\|g_{j(2)}\|_{1}), \\
  \|\tilde{g}_{j}\|_{2} & \leq & K_{j}(2) ( \|\tilde{g}_{j-1}\|_{2} +
   C_{1}\|g_{j(2)}\|_{2}).
\end{eqnarray}
Using (10.10), (10.9), and (10.7) gives
\begin{equation}
  \|g_{j(2)}\|_{2} \; \leq \; K_{j}(2) ( t_{j}^{s-(1/2)}N_{j}(2) +
   t_{j}^{k+s-2}N_{j}(k+2) ) .
\end{equation}
Since $k\geq1$ and $s=2$ the exponents are positive. If we combine
all the above, use lemma 11.1 and the ratio test, we see that
$\sum \|\tilde{g}_{j} - \tilde{g}_{j-1}\|_{1} \, < \infty$.  Hence,
$\tilde{g}_{j}\rightarrow\tilde{g}_{\infty}$ in $C^{1}(D_{\infty})$,
and we must show that it has a $C^{1}$ inverse.

From $\tilde{f}_{j}=f_{j}\circ \tilde{f}_{j-1}$, we get as above,
\begin{equation}
  |d\tilde{f}_{j}|_{\tilde{U}_{j}} \, \leq \, \prod_{i=0}^{j} ( 1 +
  |df_{i(2)}|_{\grr_{i}} ) \, \leq \, C_{1},
\end{equation}
and so $C_{1}^{-1} \leq |d\tilde{g}_{j}|_{\grr_{j+1}}  \leq C_{1}$, for all
$j$.  Thus $\tilde{g}_{\infty}$ is a $C^{1}$-diffeomorphism of $D_{\infty}$ onto
a neighborhood $\tilde{g}_{\infty}(D_{\infty})$ of $0$.

We continue the estimation of $\tilde{f}_{j}$ on a convex subdomain
$\tilde{U}_{\infty}$ of $\tilde{g}_{\infty}(D_{\infty})$ containing 0.
In the next two estimates we use the chain-rule estimate (5.5) and the
triangle inequality.  In the third we have gone back to (5.8) in order
to utilize (8.4) and (8.5).
\begin{eqnarray}
  \tilde{f}_{j} -  \tilde{f}_{j-1} & = &  f_{j(2)}\circ \tilde{f}_{j-1} , \\
 \|\tilde{f}_{j} -  \tilde{f}_{j-1}\|_{\tilde{U}_{\infty},a}  & \leq &
  K_{j}(a) ( \|f_{j(2)}\|_{\tilde{\grr}_{j}(2),a} +
  \|f_{j(2)}\|_{\tilde{\grr}_{j}(2),1}
  \|\tilde{f}_{j-1}\|_{\tilde{U}_{\infty},a}), \\
 \|\tilde{f}_{j}\|_{\tilde{U}_{\infty},a}   & \leq &
     K_{j}(a)(\|f_{j(2)}\|_{\tilde{\grr}_{j}(2),a} +
      2\|\tilde{f}_{j-1}\|_{\tilde{U}_{\infty},a}), \\
 \|f_{j(2)}\|_{\tilde{\grr}_{j}(2),a} & \leq &  K_{j}(a) (
      \|F_{j}\|_{\tilde{\grr}(2),a} +  N_{j}(3)\grd_{j}(1)N_{j}(a) ) .
\end{eqnarray}

In the case $m\in\mathbf{Z}$, we take $k<a<m+1/2$,
$a=\grl k + (1-\grl)(m+1/2)$, $0<\grl<1$, and we want to show
$\sum\|\tilde{f}_{j} -  \tilde{f}_{j-1}\|_{\tilde{U}_{\infty},a} <\infty$.
Combining (14.12) and (14.14), using (5.2), and (10.9) with $a=b=1$,
and (10.7), we are reduced to showing the following.
\begin{equation}
  \sum K_{j}(a)\|F_{j}\|_{\tilde{\grr}_{j}(2),a} \, \leq \, \sum K_{j}(a)
 (\|F_{j}\|_{\tilde{\grr}_{j}(2),k})^{\grl}
     (\|F_{j}\|_{\tilde{\grr}_{j}(2),m+1/2})^{1-\grl}  < \infty ,
\end{equation}
\begin{equation}
   \sum K_{j}(a)N_{j}(3)\grd_{j}(1)N_{j}(m+1/2) \, < \, \infty ,
\end{equation}
\begin{eqnarray}
   \sum K_{j}(a)\|f_{j(2)}\|_{\tilde{\grr}_{j}(2),1}
          \|\tilde{f}_{j-1}\|_{\tilde{U}_{\infty},m+1/2}
                       & \leq &      \\
     \sum K_{j}(a)N_{j}(3)t^{s-(1/2)}_{j}
        (1 + \|\tilde{f}_{j-1}\|_{\tilde{U}_{\infty},m+1/2}) & < &
    \,  \, \infty. \nonumber
\end{eqnarray}
The finiteness of the first two sums follows as in the last two
sections.
In fact (8.5) with $a=m+(1/2)$ and $l=m$, and (10.7) give
\begin{equation}
   \|F_{j}\|_{\tilde{\grr}_{j}(2),m+(1/2)} \; \leq \; K_{j}(m+(1/2))
  ( N_{j}(\grb)\grd_{j}(m) + N_{j}(m+\grb)t_{j}^{s} ) .
\end{equation}
Thus $\|F_{j}\|_{\tilde{\grr}_{j}(2),m+(1/2)}$ and $N_{j}(m+1/2)$ are
bounded by $B_{j}$, which has the growth (13.18).  We also see that
$\|F_{j}\|_{\tilde{\grr}_{j}(2),k}$ and $\grd_{j}(1)$ are bounded by
positive powers of $t_{j}$ times slowly growing factors.

For the third we must check the growth of the last factor.  Using
(14.15) with $a=m+(1/2)$, we get
\begin{eqnarray}
  1 + \|\tilde{f}_{j}\|_{\tilde{U}_{\infty},m+1/2} & \leq & K_{j}(m+(1/2))
  ( 1 + \|\tilde{f}_{j-1}\|_{\tilde{U}_{\infty},m+1/2} + W_{j-1} ) , \\
  W_{j-1} & \equiv  & \|F_{j}\|_{\tilde{\grr}_{j}(2),m+1/2} +
  N_{j}(3)t_{j}^{s} N_{j}(m+(1/2)) ,  \\
  W_{j} & \leq & K_{j}(m+(1/2))B_{j}.
\end{eqnarray}
The last inequality follows from (8.5) with $a=m+(1/2)$, $l=m$
and (13.9). We apply the ratio test to (14.19),
\begin{eqnarray}
  \frac{ K_{j+1}(a)N_{j+1}(3)t^{s}_{j+1}(1+\|\tilde{f}_{j}\|_{m+1/2}) }
       {  K_{j}(a)N_{j}(3)t^{s}_{j}(1+\|\tilde{f}_{j-1}\|_{m+1/2})  }
      & \leq & \hat{c}_{a}3K_{j}(3)t_{j}^{(\grk-1)s}\cdot \\
   &  & \nonumber \cdot
               ( 1+ K_{j}(m+(1/2))B_{j} ) \rightarrow 0,
\end{eqnarray}
as $j\rightarrow \infty$.

It follows that the two sequences $\tilde{f}_{j}$, $\tilde{g}_{j}$
both converge in $C^{a}$-norm, for every $a < m+1/2$, on neighborhoods
of $0$ to inverse
$C^{a}$-diffeomorphisms $\tilde{f}_{\infty}$, $\tilde{g}_{\infty}$.

The mapping $\tilde{f}_{\infty}$ provides a CR equivalence between the
original structure $X_{\gra}$ and the embedded structure $X_{\infty}$.
Hence $Z_{\infty}\circ\tilde{f}_{\infty}$ is a CR embedding of class
$C^{a}$ of our original structure, for every $a < m+1/2$ if
$m\in\mathbf{Z}$. In case $m=\infty$, the preceding can be applied,
to the same fixed sequence, for every sufficiently large integer $m$
(with perhaps different constants for each $m<\infty$).  This
finishes the proof of theorem 0.1.

In the case $m\in\mathbf{R}$, the same argument works with $m+(1/2)$
replaced by $m$.  This gives the proof of proposition 12.1.

%bbbbbbbbbbbbbbbbbbbbbbbbbbbbbbbbbbbbbbbbbbbbbbbbbbbbbbbbbbbbbbbbbbbbbbbbbbb

\noindent
Department of Mathematics

\noindent
University of Wisconsin

\noindent
gong@math.wisc.edu

\vspace{3ex}

\noindent
Department of Mathematics

\noindent
University of Chicago

\noindent
webster@math.uchicago.edu

%EEEEEEEEEEEEEEEEEEEEEEEEEEEEEEEEEEEEEEEEEEEEEEEEEEEEEEEEEEEEEEEEEEEEEEEEE
%END0000000000000000000000000000000000000000000000000000000000000000000END
\end{document}